\documentclass[11pt]{amsart}
\usepackage{amsmath,amsthm,amsfonts,amssymb}
\usepackage{hyperref}
\usepackage[all]{xy}
\usepackage{graphicx}
\usepackage{amscd}
\usepackage{graphicx}
\usepackage{pb-diagram}
\usepackage{color}
\usepackage{enumerate}

\hypersetup{
    colorlinks=true,
    linkcolor=blue,
    citecolor=green,
    urlcolor=cyan
}

\numberwithin{equation}{section}

\theoremstyle{plain} 
\newtheorem{theorem}{Theorem}[section]

\newtheorem{lemma}[theorem]{Lemma}
\newtheorem{prop}[theorem]{Proposition}

\theoremstyle{definition}
\newtheorem{definition}[theorem]{Definition}

\newtheorem{assumption}[theorem]{Assumption}

\theoremstyle{remark}
\newtheorem{remark}[theorem]{Remark}

\newcommand{\N}{\mathbb{N}}
\newcommand{\R}{\mathbb{R}}
\newcommand{\C}{\mathbb{C}}
\newcommand{\Z}{\mathbb{Z}}
\newcommand{\Q}{\mathbb{Q}}
\newcommand{\PP}{\mathbb{P}}

\newcommand{\bt}{\boldsymbol{t}}
\newcommand{\PT}{\mathbb{P}^1_{3,3,3}}
\newcommand{\PH}{\mathbb{P}^1_{2,3,6}}
\newcommand{\PQ}{\mathbb{P}^1_{2,4,4}}
\newcommand{\PB}{\mathbb{P}^1_{2,2,2,2}}
\newcommand{\ind}{\rm ind}

\newcommand{\CX}{\mathcal{X}}
\newcommand{\CY}{\mathcal{Y}}
\newcommand{\CM}{\mathcal{M}}

\newcommand{\WT}[1]{\widetilde{#1}}
\newcommand{\OL}[1]{\overline{#1}}

\newcommand{\Hom}{{\rm Hom}}

\newcommand{\CO}{\mathcal O}

\newcommand{\D}[1]{\Delta_{#1}}
\newcommand{\pt}{[\rm{pt}]}
\begin{document}

\author[Hong]{Hansol Hong}
\address{Department of Mathematics/The Institute of Mathematical Sciences	 \\ The Chinese University of Hong Kong\\Shatin\\New Territories \\ Hong Kong}
\email{hhong@ims.cuhk.edu.hk, hansol84@gmail.com}

\author[Shin]{Hyung-Seok Shin}
\address{School of Mathematics, Korea Institute for Advanced Study, 85 Hoegiro Dongdaemun-gu, Seoul 02455, Republic of Korea}
\email{hsshin@kias.re.kr}

\title[Holomorphic orbi-spheres in $\PB$]{Counting of Holomorphic Orbi-spheres in $\PB$ and Determinant Equation}
\maketitle
\begin{abstract}
We count the number of holomorphic orbi-spheres in the $\Z_2$-quotient of an elliptic curve. We first prove that there is an explicit correspondence between the holomorphic orbi-spheres and the sublattices of $\Z \oplus \Z \sqrt{-1}  (\subset \C)$. The problem of counting sublattices of index $d$ then reduces to find the number of integer solutions of the equation $\alpha \delta - \beta \gamma = d$ up to an equivalence.
\end{abstract}

\tableofcontents
\section{Introduction}
The generating functions of Gromov-Witten invariants (GW potentials in short) on 1-dimensional Calabi-Yau orbifolds were shown to be quasi-modular forms (or their generalizations for some higher dimensional Calabi-Yau). For example
Milanov-Ruan \cite{MR} proved the quasi-modularity
for elliptic orbifold projective lines $\PT$, $\PQ$ and $\PH$, and later Shen and Zhou \cite{SZ} proved it for all one-dimensional compact Calabi-Yau that includes $\PB$.
(See also \cite{LZ} for the modularity of open GW potentials of these orbifolds.)

It draws our attention that one can observe a number theoretic phenomenon in the curve counting problem for elliptic orbifolds, and this paper as well as our preceding work \cite{HS} aims to see this more straightforwardly. In \cite{HS}, we studied classification of  holomorphic orbi-spheres which contribute to (small) quantum product of three elliptic orbifold projective lines $\PT$, $\PQ$ and $\PH$. While the full GW potentials had previously been computed making use of algebraic techniques such as WDVV-equations (see \cite{KS} for potentials of $\PT$, $\PQ$ and $\PH$ for all genera, and \cite{ST} for genus-0 potentials of $\PT$ and $\PB$), our method in \cite{HS} was based on elementary counting so that it reveals more directly a combinatorial feature of the GW invariants for those elliptic orbifolds. 

Interestingly, there is a correspondence between those invariants (3-point correlators in Gromov-Witten invariants) and the number of integer solutions of certain quadratic Diophantine equations.
Consequently, we were able to show that the structure constants for quantum products for $\PT,\PQ$ and $\PH$ are theta functions in K\"ahler parameter $q$, which directly implies their modularities.

In this short article, we classify holomorphic orbi-spheres in $\PB$ applying the similar technique as in \cite{HS}. We emphasize that our goal is not to compute the GW potential itself, but rather, to examine an elementary, yet more geometric technique to count holomorphic spheres that contribute to the GW potential.  We will see that the counting  problem reduces to a simple combinatorial question, that is finding the number of integer solutions for certain equations.

The main difference from the previous work \cite{HS} is as follows.
As we were only interested in 3-point correlators (for quantum products) in \cite{HS}, we fixed the complex structure of the domain orbifold Riemann surface using equivalence of stable maps (by fixing the position of three orbifold singular points of the domain).
On the other hand, the domain orbifolds in the case of $\PB$ are orbi-spheres with four cone-type singularities of order $2$ equipped with a complex structure, and there is  a family of such domains which are not bi-holomorphic to each other.
In fact, the moduli space of the domain orbifolds is isomorphic to the moduli of elliptic curves, which can be described in terms of $\Z$-lattices in $\C$.

Combining this description of the moduli of domains and the lifting theory of {\it orbi-map}s to universal covering spaces of $\PB$ (which is  simply $\C$), we shall prove that the curve-counting problem is again equivalent to finding the number of integer solutions (up to gauge) of some quadratic equation. Here is our main theorem in this paper.

\begin{theorem}
The 4-point non-zero degree $d$ Orbifold Gromov-Witten invariant of $\PB$ is the count of integer solutions $(\alpha,\beta,\gamma,\delta)$ of
$$ \alpha \gamma - \beta \delta = d$$
up to the linear action of ${\rm SL} (2,\Z)$.
\end{theorem}

In fact, the number of solutions of the above equation only gives the sum of all contributions to the Gromov-Witten invariants with fixed degree $d$, but the precise Gromov-Witten invariant can be easily deduced from this after elementary combinatorics to be made in Section \ref{Sec:computation}.
More rigorous statement of the theorem will be shown in Section \ref{sec:mainsec_sphcount} (see Theorem \ref{thm:lumpsum}).

The organization of the paper is as follows.
Section \ref{sec:prelim} gives a brief review on the orbifold Gromov-Witten theory and the computational result on $\PB$ in \cite{ST}.
The main part of the paper is Section \ref{sec:mainsec_sphcount} where we classify the holomorphic orbi-spheres in $\PB$ and study their equivalence relations.
In Section \ref{Sec:computation}, we compute 4-points Gromov-Witten invariants of $\PB$ precisely, using the classification result in Section \ref{sec:mainsec_sphcount}.
In Appendix \ref{app:trans}, we prove regularity of the moduli space of non-trivial holomorphic orbi-spheres contributing to the 4-point Gromov-Witten invariants of $\PB$, and
Appendix \ref{app:orbcovth} gives a brief review on orbifold covering theory.

\subsection*{Notation}
Throughout the article, we will mainly deal with the holomorphic maps $u : \PB(\Lambda) \to \PB(\Lambda_0)$ between two $\PB$'s but with different complex structures.
\begin{equation*}
\begin{array}{ll}
x_j :& \mbox{orbi-points in the target}\,\, \PB (\Lambda_0)\\
y_j :& \mbox{orbi-points in the domain}\,\, \PB (\Lambda)\\
z_j :& \mbox{orbi-marked points in the domain}\,\, \PB (\Lambda).
\end{array}
\end{equation*}
($ \{ y_1,y_2,y_3,y_4\} = \{z_1,z_2,z_3,z_4\}$ but $y_i \neq z_i$ in general.)

\subsection*{Acknowledgement}
We express our gratitude to Cheol-Hyun Cho who brought our attention to the subject. The first author thanks Naichung Conan Leung and Yalong Cao for comments and references.
The second author thanks Myungho Kim, Hyenho Lho, and Jeongseok Oh for valuable discussions.

\section{Preliminaries}\label{sec:prelim}
In this section, we give a precise definition of the orbifold $\PB$ which will be studied throughout the article, and briefly review the orbifold Gromov-Witten theory developed by Chen and Ruan.

\subsection{The orbifold representation of $\PB$}

Topologically $\PB$ is a sphere with four cone points each of which has $\Z_2$-singularity. (It is sometimes called ``pillowcase" by its shape. See (b) of Figure \ref{fig:2222intro}.) In this section, we give a concrete description of its orbifold structure for later use, and fix it to be the orbifold representation of $\PB$ throughout the article.

Let $E_0$ be an elliptic curve which is associated with the integer lattice $\Lambda_0 := \Z \cdot 1 + \Z \cdot \sqrt{-1}$ in $\R^2$. Observe that $\Lambda_0$ is invariant under a holomorphic involution $\sigma : z \mapsto -z$ on $\C$. Therefore $E_0$ admits an action of $\Z_2$ generated by $\sigma$. We finally define $\PB$ to be the quotient orbifold $[E_0 / \Z_2]$ which inherits a complex structure from $E_0$. We denote this complex structure by $J_{std}$, and fix it once for all.
\begin{remark}\label{rmk:jonpb}
Any elliptic curve $\C / \Lambda$ admits the same type of (holomorphic) $\Z_2$-action. Its quotient is still diffeomorphic to $\PB$, but not biholomorphic to $(\PB,J_{std})$ in general.
The induced complex structure on $\PB$ from the elliptic curve therefore depends on the choice of the lattice $\Lambda$. We will revisit this issue in \ref{subsec:cpxstrdomain}.
\end{remark}

Consider the universal covering map $\C \to  E_0(=\C / \Lambda_0 )$. The composition with the quotient map $E_0 \to [E_0 / \Z_2]$ gives an orbifold (universal) covering
$$p : \C \to E_0 \to [E_0 / \Z_2].$$
(See for e.g. \cite{Thu} for detailed explanation on orbifold covering theory.)
Fundamental domain of $E_0$ and $\PB$ in the universal cover $\C$ is depicted in (a) of Figure \ref{fig:2222intro}.

\begin{figure}
\begin{center}
\includegraphics[height=2in]{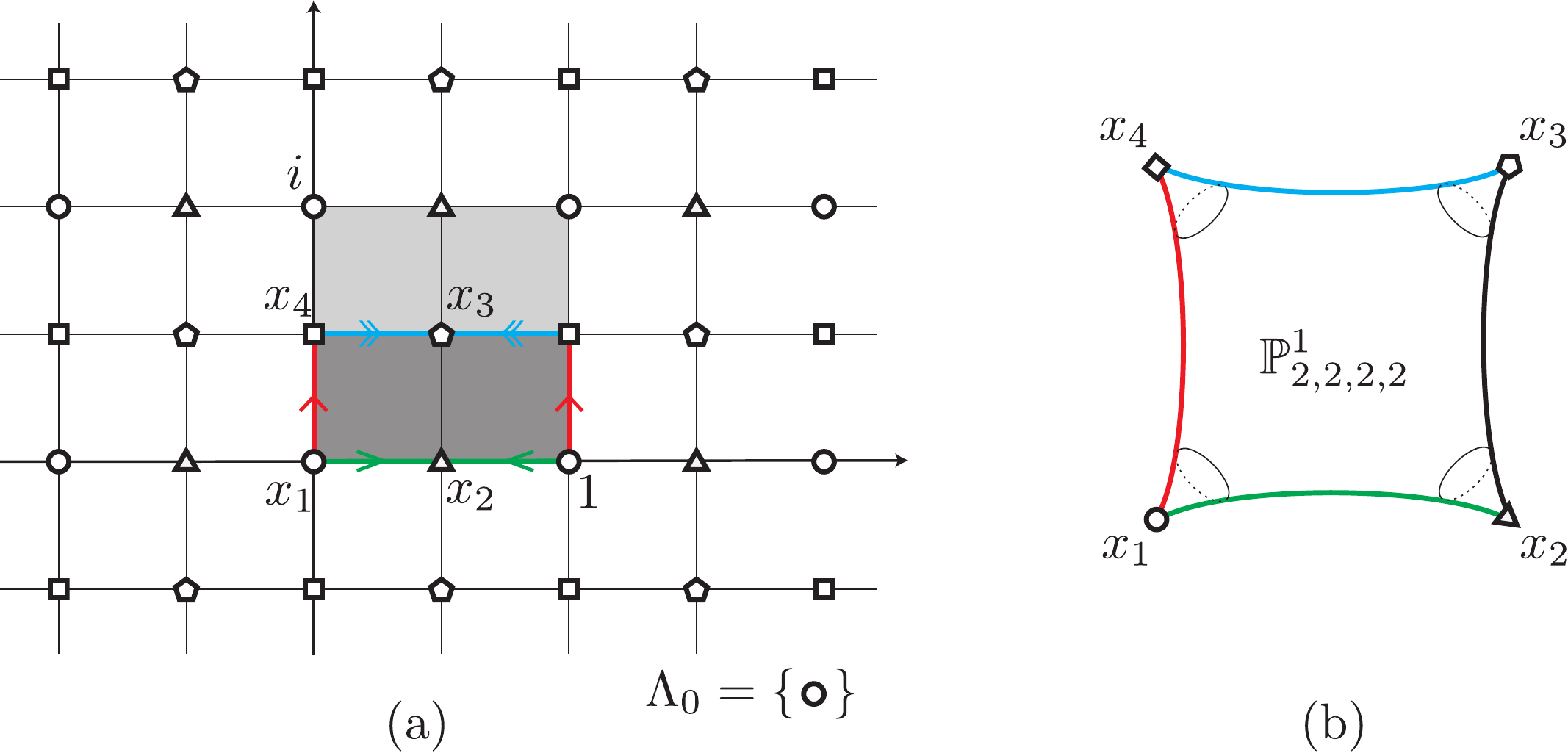}
\caption{(a) Fundamental domain of $E_0$ and $\PB$ in $\C$ and (b) $\PB$}\label{fig:2222intro}
\end{center}
\end{figure}

We denote four orbifold points in $\PB$ by $x_1, x_2, x_3, x_4$, which are arranged as follows:
\begin{equation}\label{eqn:arrangementxi}
\begin{array}{l l}
p^{-1}(x_1) = \Lambda_0,&
p^{-1}(x_2) = \frac{1}{2} + \Lambda_0,\\
p^{-1}(x_3) = \frac{1+\sqrt{-1}}{2} + \Lambda_0,&
p^{-1}(x_4) = \frac{\sqrt{-1}}{2} + \Lambda_0.
\end{array}
\end{equation}
One can easily see that they correspond to four fixed points for the $\Z_2$-action on $E_0$.

The orbifold cohomology ring $H^*_{orb} (\PB, \Q)$ is generated by cycles in twisted sectors as well as those in $|\PB|\cong S^2$, which we denote by
\begin{equation}\label{eqn:PBHbasis}
\left\{\mathbf{1}, \D{1}, \D{2}, \D{3}, \D{4}, \pt \right\}.
\end{equation}
Here, $\mathbf{1}$ ($\deg=0$) and $\pt$ ($\deg=2$) are the Poincar\'e dual of the fundamental class and the point class in the smooth sector respectively, and the others come from twisted sectors at four $\Z_2$-singular points.

$H^*_{orb} (\PB, \Q)$ admits a ring structure via Chen-Ruan cup product ``$\cup$" which preserves degrees. We refer readers to \cite{CR1} for general theory of orbifold cohomology.
The cohomology class $\D{j}$ is a unique $\deg=1$ element of $H^1_{orb}(\PB, \Q)$ which is supported at the orbifold point $x_j$ and satisfies
$$\D{j} \cup \D{j} = \frac{1}{2} \pt.$$
In addition, $\mathbf{1} \in H^*_{orb} (\PB, \Q)$ serves as the unit for the cup product, and $\D{i} \cup \D{j} =\D{i} \cup \pt =0$ for $i \neq j$, which completes the multiplication table among generators.

\subsection{Orbifold Gromov-Witten invariants}
We briefly recall orbifold Gromov-Witten theory from \cite{CR2}, mainly focusing on our example $\PB$. We will also review the result of Satake-Takahashi \cite{ST} on the Gromov-Witten invariant of $\PB$ at the end of the section.

Let $(X,\omega)$ be an compact effective symplectic orbifold with a compatible almost complex structure $J$. The orbifold Gromov-Witten invariants count the number of ``marked orbifold stable maps", which are holomorphic maps from an orbifold Riemann surface $(\Sigma,j)$ (with a complex structure $j$) into $(X,J)$ satisfying certain intersection condition at each marking. For such a map $f : \Sigma \to X$, we require that the orbifold singularities at $x \in \Sigma$ and at $f(x) \in X$ are compatible in the following sense: consider the local lifting $\WT{f}_{x}: (V_x, G_x) \to (V_{f(x)}, G_{f(x)})$ of $f$. Then the homomorphism
\begin{equation}\label{eqn:injpi1}
\pi_1^{orb} ([V_x/G_x]) =G_x \hookrightarrow \pi_1^{orb} ([V_{f(x)} / G_{f(x)}])=G_{f(x)}
\end{equation}
is injective. (See \cite[Definition 2.3.3]{CR2} for more details.)

Two orbifold stable maps are equivalent to each other if one is the reparameterization of the other by a biholomorphism between domains which respect all relevant data such as marked points.
For a given a homology class $\beta \in H_2(|{X}|; \Z)$, we denote by $\OL{\CM}_{g,k,\beta}({X}, J)$ the (compactified) moduli space of equivalence classes of orbifold stable maps of genus $g$, with $k$ marked points, and of homology class $\beta$.

 We fix a $\Q$-basis $\{\gamma_i \}_{i=1,\cdots,N}$ of $H^\ast_{orb} ({X}, \Q)$.
 Then the $k$-fold Gromov-Witten invariants is defined by the following equation:
\begin{equation*}\label{eq:nGW}
\langle \gamma_1 , \cdots, \gamma_k \rangle^{{X}}_{g, k, \beta} := \int_{[\OL{\CM}_{0,k,\beta}({X})]^{vir}} ev_1^* \gamma_1 \wedge \cdots \wedge ev_k^* \gamma_k.
\end{equation*}
where $ev_i$ is the evaluation of the stable map at $i$-th marked point.
We also define $\langle \gamma_1 , \cdots, \gamma_k \rangle^{{X}}_{g, k}$ to be the weighted sum $\sum_{\beta} \langle \gamma_1 , \cdots, \gamma_k \rangle^{{X}}_{g, k, \beta} q^{\omega (\beta)}$.

\begin{remark}
$\OL{\CM}_{0,k,\beta}({X})$ admits a virtual fundamental class $[\OL{\CM}_{0,k,\beta}({X})]^{vir}$ by various constructions, which will not be needed in our case of  $\PB$ since the moduli space is already smooth.
See Appendix \ref{app:trans}.
\end{remark}

If we set $\bt := \sum t_i \gamma_i$, then the generating function for the Gromov-Witten invariants is defined as
\begin{equation*}
F^{{X}}_0 (\bt):= \sum_{k, \beta} \frac{1}{k!} \langle \bt, \cdots, \bt \rangle^{{X}}_{0, k, \beta} q^{\omega(\beta)},
\end{equation*}
which we will call the genus-0 Gromov-Witten potential for ${X}$.

Now let us examine the case of $\PB$ more concretely. In this case, we have four variables $t_i$ associated with cycles in twisted sectors $\D{i}$ for $i=1,2,3,4$ (see \eqref{eqn:PBHbasis}) together with $t_0$ corresponding to the unit class $\mathbf{1}$.
Instead of using variable $t_5$ for the point class, we use $q$ for $q=\exp(t_5)$. This is due to divisor axiom for the Gromov-Witten invariant.
Thus the genus zero Gromov-Witten potential $F$ of $\PB$ is a power series in $t_i$ ($0 \leq i \leq 4$) whose coefficients are also power series in $q$, the (exponentiated) symplectic area of $\PB$.
Satake-Takahashi \cite{ST} computed $F(q,t_i)$ completely by WDVV method, which is given as follows:
\begin{equation}\label{eqn:STcomPB}
\begin{array}{ll}
F = & \frac{1}{2} t_0^2 \log q + \frac{1}{4} t_0 ( t_1^2 + t_2^2 + t_3^2 + t_4^2 ) \\ &+ ( t_1 t_2 t_3 t_4 ) \cdot f_0 (q) + \frac{1}{4} ( t_1^4 + t_2^4 + t_3^4 + t_4^4 ) \cdot f_1 (q) \\ &+ \frac{1}{6} ( t_1^2 t_2^2 + t_1^2 t_3^2 + t_1^2 t_4^2 + t_2^2 t_3^2 + t_2^2 t_4^2 + t_3^2 t_4^2 ) \cdot f_2 (q)
\end{array}
\end{equation}
where
\begin{equation*}\label{eqn:f0f1f2f}
\begin{array}{ll}
f_0 (q) &:= \frac{1}{2} (f(q) - f(-q))\\
f_1 (q) &:= f(q^4)\\
f_2 (q) &:= f(q) - f_0(q) - f_1 (q)\\
f (q) &:= -\frac{1}{24} + \sum_{n=1}^{\infty} n \frac{q^n}{1 - q^n}
\end{array}
\end{equation*}
This article aims to reinterpret the coefficients $f_i(q)$ (and $f(q)$) in terms of holomorphic sphere countings in $\PB$. Note that nonconstant terms in $f(q)$ can be written as
\begin{equation}\label{eqn:nonconstf}
\sum_{n=1}^{\infty} n \frac{q^n}{1 - q^n} = \sum_{n=1}^{\infty} n\, q^n \left( \sum_{k=0}^{\infty} q^{kn} \right) = \sum_{n,k \geq 1}^{\infty} n q^{nk}= \sum_{d=1}^{\infty} \mathfrak{D}(d) q^d
\end{equation}
for $\mathfrak{D}(d)$ a divisor sum function, which will turn out to be closely related to the combinatorial nature of our sphere countings later. (In literatures, $\sigma_1(d)$ is a more common notation for the divisor sum function.) First few terms of $f$ are given as
\begin{align*}
f (q) &= -\frac{1}{24} + q + 3 q^2 + 4 q^3 + 7 q^4 + 6 q^5 + 12 q^6 + 8 q^7 + 15 q^8 + 13 q^{9} \\
&+ 18 q^{10} + 12 q^{11} + 28 q^{12} + 14 q^{13} + 24 q^{14} + 24 q^{15} + 31 q^{16} + 18 q^{17} + \cdots.
\end{align*}

\section{Holomorphic orbi-spheres countings in $\PB$}\label{sec:mainsec_sphcount}

In this section, we prove that the counting problem of degree $d$ holomorphic orbi-spheres in $\PB$ (with four orbi-marked points) is equivalent to finding integer solutions $(\alpha, \beta, \gamma,\delta)$ of the equation
\begin{equation}\label{eqn:mainadbcN1}
\alpha \delta - \beta \gamma = d
\end{equation}
up to $SL (2, \Z)$-action on $\left\{ \left(\begin{array}{cc}\alpha & \gamma \\\beta & \delta\end{array}\right) : \alpha, \beta,\gamma,\delta \in \Z  \right\}$, which we will call the {\it determinant equation}. For this, we will turn the counting problem into the classification of sublattices of $\Lambda_0$ (the standard $\Z^2$-lattice in $\C$).
In the terminology of Gromov-Witten theory, the number of solutions of \eqref{eqn:mainadbcN1} modulo equivalence provides the computation of the following type invariants:
\begin{equation}\label{eqn:mainGW}
 \langle \D{i}, \D{j}, \D{k}, \D{l} \rangle_{0,4,d}\quad 1 \leq i,j,k,l \leq 4 ,\,\, d \geq 1.
\end{equation}
Here $d$ indicates the degree of contributing holomorphic orbi-spheres (regarded as a continuous map from a sphere to itself).
It will turn out in the counting procedure below that only nontrivial are  $\langle \D{1}, \D{2}, \D{3}, \D{4} \rangle_{0,4,d}$, $\langle \D{j}, \D{j}, \D{j}, \D{j} \rangle_{0,4,d}$, $\langle \D{j}, \D{j}, \D{k}, \D{k} \rangle_{0,4,d}$
for $1 \leq j < k \leq 4$ up to reordering of $\D{\bullet}$'s.
The complete computation the genus-0 Gromov-Witten potential will be given in the Section \ref{Sec:computation}.

We first clarify one delicate issue of the counting problem we will consider, that is, complex structures of the domain orbi-curve can vary.

\subsection{Complex structures on the domain orbi-curve}\label{subsec:cpxstrdomain}


As we are considering holomorphic orbi-spheres for \eqref{eqn:mainGW}, the domain of such maps $u$ are topologically $\PB$ itself. Hence, the map $u$ we are interested in is a holomorphic map from $\PB$ to itself, sending orbi-points to orbi-points. Locally around an orbi-point, $u$ is given by the following form 
\begin{equation}\label{eqn:locorbmor}
 \C/\Z_2 \stackrel{[id]}{\longrightarrow} \C/ \Z_2
\end{equation}
where $[id]$ is induced by $\C \stackrel{id}{\to} \C$.

Important point is that complex structure of the domain can be arbitrary while the complex structure of the target of $u$ is fixed by $J_{std}$. Therefore, we should look into the holomorphic maps $u$
$$ u : (\PB, j) \to (\PB, J_{std})$$
for all possible complex structures $j$ on $\PB$.

\begin{remark}
In general, there could be degenerate maps (nodal curves) in $\WT{\mathcal{M}}_{0,4,d} (\D{i},\D{j},\D{k},\D{l})$ but we shall rule out such possibility in \ref{subsec:nodalcontri} by a simple topological argument.
\end{remark}

We now investigate possible complex structures that the domain $\PB$ can admit. Recall that $\PB$ is a $\Z_2$-quotient of the 2-dimensional torus $T^2$ where we identify $T^2$ with $\C / \Lambda$ and the $\Z_2$-action is generated by $z \mapsto -z$ (see Remark \ref{rmk:jonpb}).
Any complex structure $j$ on $\PB$ induces a $\Z_2$-invariant complex structure on a torus $T^2$ which is a pull-back of $j$ by the quotient map $T^2 \to \PB$.
Conversely, any complex structure on $T^2$ is $\Z_2$-invariant in the following sense:
\begin{itemize}
\item[(i)]
the quotient space $T^2= \C / \Lambda$ can be naturally equipped with the complex structure inherited from $\C$ (we denote the resulting elliptic curve by $E_\Lambda$), and
\item[(ii)]
two such elliptic curves
$E_\Lambda=\C / \Lambda$ and $E_{\Lambda'}=\C / \Lambda'$ are biholomorphic if and only if $\Lambda = \alpha \Lambda'$ for some $\alpha \in \C^{\times}$.
\item[(iii)]
Since any $\Z$-lattice $\Lambda$ is invariant under the action $z \mapsto -z$ on $\C$, so is the complex structure on $\C / \Lambda$.
\end{itemize}

Consequently, we can identify the set of complex structures on $\PB$ with the set of those on $T^2$.
From now on, we denote by $j_{\Lambda}$ the complex structure on $\PB$  inherited from $\C / \Lambda$.
Taking complex structure into account, we have $(\PB , j_{\Lambda}) = [E_{\Lambda}/\Z_2]$ as complex orbifolds.

\begin{definition}
We denote by $\PB(\Lambda)$ the complex orbifold $[E_{\Lambda}/\Z_2]$ equipped with the complex structure $j_{\Lambda}$. In particular, the target $\PB$ (of holomorphic orbi-spheres) with $J_{std}$ is $\PB(\Lambda_0)$ for the standard lattice $\Lambda_0 := \Z^2 \subset \C$.
\end{definition}
Note that $\PB(\Lambda)$ and $\PB(\alpha \Lambda)$ ($\alpha \in \C^\times$) are biholomorphic from the above discussion.

\bigskip
For $u : \PB(\Lambda) \to \PB(\Lambda_0)$, there are four orbi-marked points $z_1$, $z_2$, $z_3$, $z_4$ in the domain $\PB(\Lambda)$ which are the image of $\frac{1}{2} \Lambda$. For a basis $\{w_1, w_2\}$ of $\Lambda$, $\{0, v_1/2, v_2/2, v_1+v_2 /2\} \subset \C$ project down to four orbi-marked points of $\PB(\Lambda)$ (see Figure \ref{fig:sublat}). We denote these four orbifold points as
\begin{equation}\label{eqn:orbyi}
y_1 = 0,\quad y_2 = \frac{v_1}{2}, \quad y_3= \frac{v_1 + v_2}{2},\quad y_4=\frac{v_2}{2} \mod \Lambda.
\end{equation}
(Then the orbi-marked points $z_i$ in the domain is one of $y_j$'s.) Later, we will choose a specific basis to fix $y_j$'s. Note that $y_1=0$ does not depend on the choice of basis.

Recall from \eqref{eqn:arrangementxi} that the orbi-points $x_i$ in the target $\PB(\Lambda_0)$ admits a similar description.
\begin{figure}[h]
\begin{center}
\includegraphics[height=2.2in]{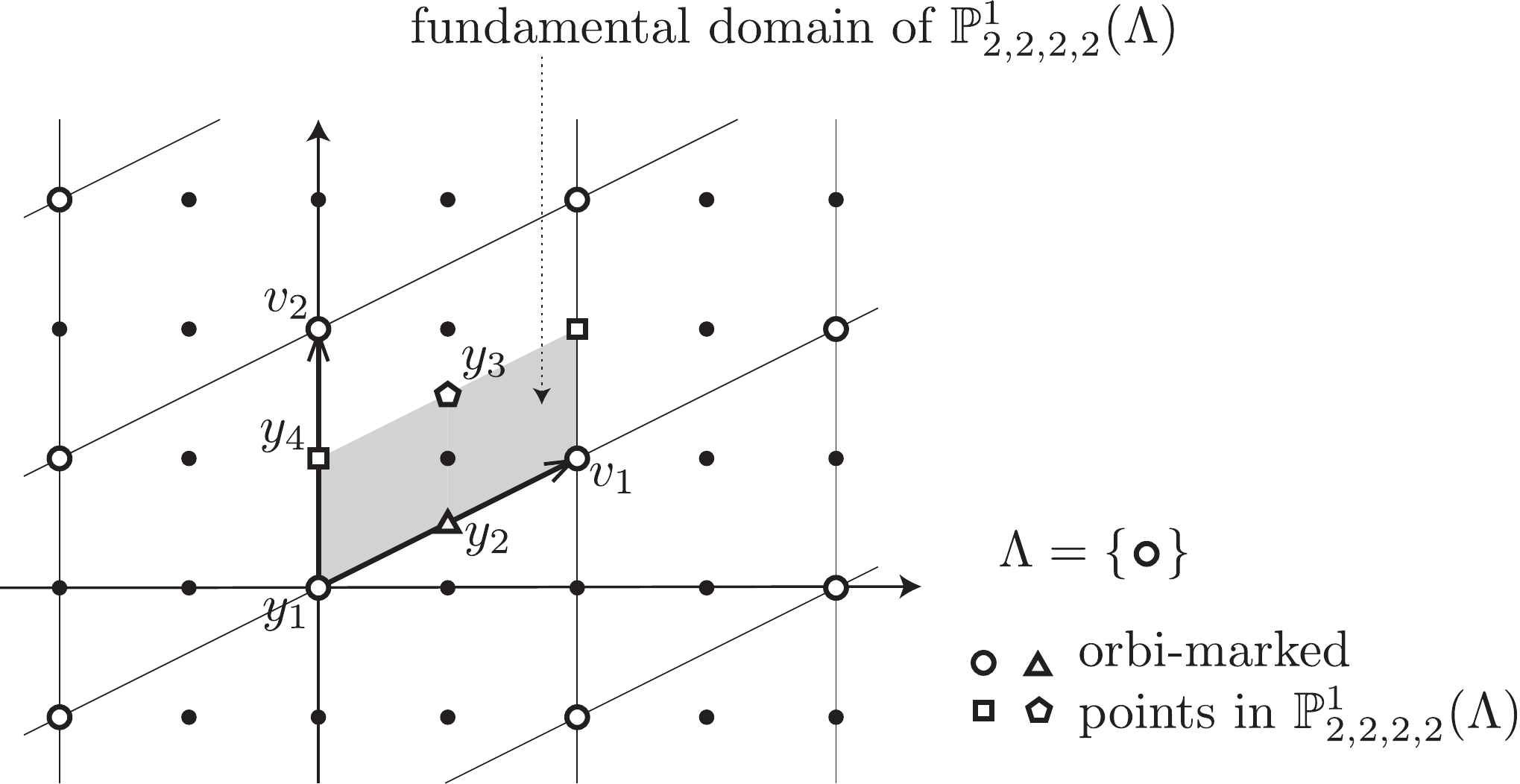}
\caption{Fundamental domain of $\PB (\Lambda)$ in $\C$ for $\Lambda=\langle (2,1),(0,1) \rangle$}\label{fig:sublat}
\end{center}
\end{figure}
We make the following assumption for simplicity.
\begin{assumption}\label{assm:firstmark}
Without loss of generality, we may assume the first marked point $z_1$ to be $y_1$ which is the image of $0 \in \C$ under the map $\C \to \PB(\Lambda)=[ E_\Lambda / \Z_2]$. (The same is true for $x_1$ by our arrangement before. See \eqref{eqn:arrangementxi}.) \end{assumption}

As mentioned in the beginning of the section, we will count the holomorphic orbi-spheres of the type $\langle \D{i}, \D{j}, \D{k}, \D{l} \rangle_{0,4,d}$.
We now provide a concrete description of the corresponding moduli spaces.

\begin{definition}\label{def:Mifjkl}
We define
\begin{equation*}
\begin{array}{l}
\WT{\mathcal{M}}_{0,4,d} (\D{i},\D{j},\D{k},\D{l}) :=\\
\left\{\left( \PB (\Lambda) \stackrel{u}{\to} \PB(\Lambda_0),\,z_i \right)  \left|
\begin{array}{l}
 \bullet \,\, u\,\,\mbox{is}\,\,(j_{\Lambda}, J_{std})\mbox{-holomorphic},\, \deg u =d\\
 \bullet \,\, \{z_1,z_2,z_3,z_4\} = \left[\frac{1}{2} \Lambda\right],\,\, z_1 = [0] \\
 \bullet \,\, u : (z_1,z_2,z_3,z_4) \mapsto (x_i,x_j,x_k,x_l)\\
 \bullet \,\, \mbox{local lifting of}\,\,u\,\,\mbox{at}\,\,z_i\,\,\mbox{is given as in \eqref{eqn:locorbmor}}
\end{array}
\right. \right\}.
\end{array}
\end{equation*}
Two maps $\left( \PB (\Lambda) \stackrel{u}{\to} \PB(\Lambda_0),\,z_i \right)$ and $\left(\PB (\Lambda') \stackrel{u'}{\to} \PB(\Lambda_0),\,z_i' \right)$ in $\WT{\mathcal{M}}_{0,4,d} (\D{i},\D{j},\D{k},\D{l})$ are said to be equivalent if there exits a biholomorphism $\phi : \PB (\Lambda) \to \PB (\Lambda')$ such that $u = u' \circ \phi$ and $\phi (z_i) = z_i'$.
Denote the resulting equivalence relation by $\sim$.

Finally, the moduli space $\mathcal{M}_{0,4,d} (\D{i},\D{j},\D{k},\D{l})$ is defined by the quotient of $\WT{\mathcal{M}}_{0,4,d} (\D{i},\D{j},\D{k},\D{l})$ via the equivalence relation $\sim$.

\end{definition}

The following proposition asserts that $\mathcal{M}_{0,4,d} (\D{i},\D{j},\D{k},\D{l})$ is the correct moduli space for the Gromov-Witten invariant we are interested in.

\begin{prop}\label{prop:regnodeg}
$\mathcal{M}_{0,4,d} (\D{i},\D{j},\D{k},\D{l})$ is smooth, compact and of dimension 0. Furthermore
$$\langle \D{i},\D{j},\D{k},\D{l} \rangle_{0,4,d} = | \mathcal{M}_{0,4,d} (\D{i},\D{j},\D{k},\D{l}) |,$$
where $| \,\cdot\, |$ means the cardinality of the set.
\end{prop}

\begin{proof}
The proof follows from classification result in \ref{subsec:nodalcontri} and the regularity in Appendix \ref{app:trans}.
\end{proof}

\subsection{Linear liftings of holomorphic orbi-spheres in $\PB$ and sublattices of $\Lambda_0$ in $\C$}

We next relate holomorphic orbi-spheres $u : \PB(\Lambda) \to \PB(\Lambda_0)$ in $\WT{\mathcal{M}}_{0,4,d} (\D{i},\D{j},\D{k},\D{l})$ with sublattices of $\Lambda_0$ in $\C$.

To make the exposition simpler, we fix the intersection condition at the first marked point $z_1(=y_1$; Assumption \ref{assm:firstmark}) of $u$ to be $\D{1}$ (supported at $x_1$), which does not violate generality. So, $z_1$ should be mapped to $x_1$.

\begin{definition}
We write
\begin{equation}\label{eqn:tempM}
\begin{array}{l}
\WT{\mathcal{M}}_{0,4,d} (1) := \bigcup_{j,k,l} \WT{\mathcal{M}}_{0,4,d} (\D{1} , \D{j}, \D{k}, \D{l}),\\
\mathcal{M}_{0,4,d} (1) := \bigcup_{j,k,l} \mathcal{M}_{0,4,d} (\D{1} , \D{j}, \D{k}, \D{l})
\end{array},
\end{equation}
where $(1)$ in the left hand side stands for the first insertion being $\Delta_1$.
\end{definition}
We will soon see that several components in the right hand side of \eqref{eqn:tempM} are actually zero.

In what follows, we will study the covering theory associated with a map $u$ in $\WT{\mathcal{M}}_{0,4,d} (1) $. $z_1(=y_1)$ and $x_1$ will denote the base points in the domain and the target of $u$, respectively.

\begin{remark}\label{rmk:234s3}
We do not fix the locations of other markings $z_2, z_3, z_4$ so that the same map $u$ can be interpreted as several different elements in \eqref{eqn:tempM}  depending on the position of $z_2,z_3,z_4$. i.e. $(z_2,z_3,z_4) = (y_{\tau(2)},y_{\tau(3)},y_{\tau(4)})$ for some permutation $\tau$ on $\{2,3,4\}$ (see \eqref{eqn:orbyi}).
\end{remark}

Orbifold covering theory will be briefly reviewed in Appendix \ref{app:orbcovth}.
For readers who do not want to go into details, we remark that we will only need the lifting theorem in orbifold covering theory which works similarly as in ordinary case, and the uniqueness of the universal covering space.

\begin{figure}[h]
\begin{center}
\includegraphics[height=2in]{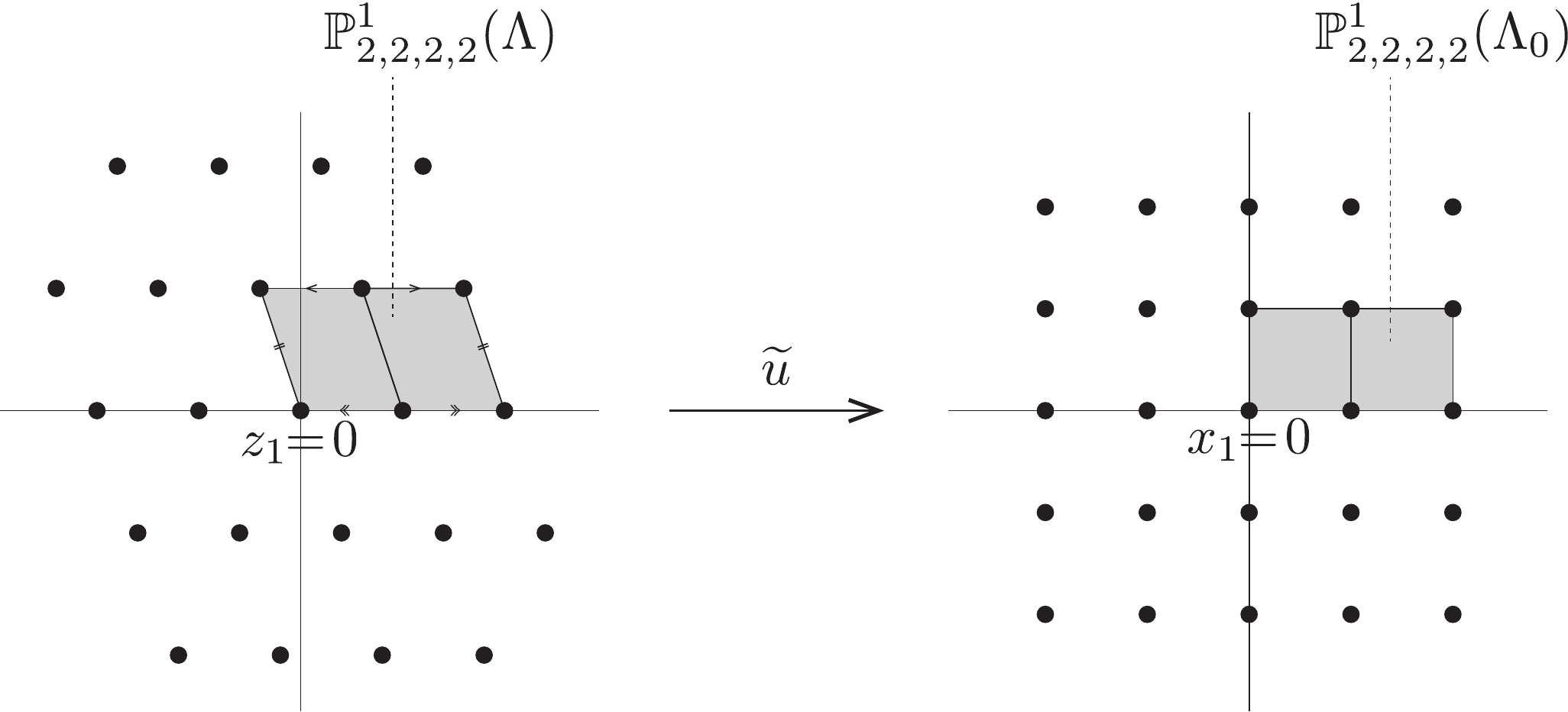}
\caption{Universal covers of $\PB (\Lambda)$ and $\PB (\Lambda_0)$}\label{fig:univcoverz1x1}
\end{center}
\end{figure}

Given an element $u : (\PB(\Lambda), z_1) \to (\PB(\Lambda_0),x_1)$ of $\WT{\mathcal{M}}_{0,4,d}(1)$, we first take universal covering spaces of the domain and the target of $u$:
\begin{equation}\label{eqn:twounivcov2}
\pi: (\C,0) \to (\PB(\Lambda),z_1 ),\quad p:(\C,0) \to (\PB(\Lambda_0),x_1 )
\end{equation}
where $0$ is the base point of both universal covers (see Figure \ref{fig:univcoverz1x1}). Here, we fix the covering map to be the composition of the quotient maps
$$\pi : \C \longrightarrow E_\Lambda (= \C / \Lambda) \longrightarrow \PB(\Lambda) (= [E_\Lambda / \Z_2]),$$
as in the construction of $\PB(\Lambda)$, and similar for $p$.
Note that two universal covering spaces in \eqref{eqn:twounivcov2} carry natural lattice structures, which are $\Lambda$ and $\Lambda_0$ respectively.

We claim that the lifting of $u$ on the level of universal cover gives a ($\C$-)linear map sending $\Lambda$ to a sublattice of $\Lambda_0$. We will need the following lemma which is crucial to apply the uniqueness of the universal cover below.
\begin{lemma}\label{lem:RHformula}
$u : (\PB (\Lambda),z_1) \to (\PB (\Lambda_0),x_1)$ in $\WT{\mathcal{M}}_{0,4,d}(1)$ is an orbifold covering map.
\end{lemma}

The proof of the lemma will be given in Appendix \ref{app:orbcovth}.

\begin{prop}\label{prop:lifting}
Any element $u : (\PB (\Lambda),z_1) \to (\PB (\Lambda_0),x_1)$ in $\WT{\mathcal{M}}_{0,4,d}(1)$
can be lifted to a linear map $\WT{u} : \C \to \C$ sending $\Lambda$ into $\Lambda_0$. In particular, $u$ uniquely determines a sublattice $\Lambda_u:=\WT{u} (\Lambda)$ of $\Lambda_0$.
\end{prop}
\begin{proof}
Applying the lifting theorem to the orbifold covering map $u \circ \pi$, there is a unique holomorphic map $\WT{u}$ such that the following diagram commutes:
\begin{equation*}\label{diag:lifting1}
\xymatrix{
(\C,0) \ar[rr]^{\WT{u}} \ar[d]^{\pi}	&&	(\C,0) \ar[d]^{p}	\\
(\PB(\Lambda),z_1)	\ar[rr]^{u}	&&	(\PB(\Lambda_0), x_1)
}.
\end{equation*}
Since $u$ is a covering map by Lemma \ref{lem:RHformula} and the composition of two covering maps is again a covering map,  $u \circ \pi $ gives the (orbifold) universal cover of $\PB(\Lambda_0)$.
By the uniqueness of the universal cover, two universal covers $u \circ \pi $ and $p$ should be isomorphic, which implies that $\WT{u}$ is an biholomorphism. By elementary complex analysis, such a biholomorphism reduces to a degree 1 polynomial, but the constant term of $\WT{u}$ must vanish in order to preserve the base point $0$ of $\C$. Therefore,
$\WT{u} (z) = \xi z$ for some $\xi \in \C^{\times}$.
Observe that $\pi^{-1} (z_1) = \Lambda$ and $p^{-1} (x_1) = \Lambda_0$, Since $\WT{u}$ is a lifting of $u$, we see that $\Lambda_u:= \xi \cdot \Lambda$ is a sublattice of $\Lambda_0$.
\end{proof}

\begin{definition}\label{def:indindlat}
\begin{enumerate}
\item
For $u$ in $\WT{\mathcal{M}}_{0,4,d} (1)$, we call the linear map $\WT{u}$ on $\C$ in Proposition \ref{prop:lifting}  {\em the  linear lifting of $u$}.
\item
For a 2-dimensional sublattice $\Lambda$ of $\Lambda_0$, $\Lambda_0 / \Lambda$ is a finite abelian group. We define $\ind (\Lambda)$ (the index of $\Lambda$) by the number of elements in $\Lambda_0 / \Lambda$.
\item
We denote the set of 2-dimensional sublattices of $\Lambda_0$ with index $d$ by $\mathcal{L}_d (\Lambda_0)$.
\end{enumerate}
\end{definition}
The index of a sublattice $\Lambda$ of $\Lambda_0$ can be alternatively described as follows. Let $\{(\alpha,\beta), (\gamma,\delta)\}$ be a $\Z$-basis of $\Lambda$ such that $(\alpha,\beta)=\alpha+\beta \sqrt{-1}$ and $(\gamma,\delta) = \gamma+ \delta \sqrt{-1}$ are positively oriented. Then it is an easy exercise to check that
\begin{equation}\label{eqn:degdet}
\ind(\Lambda) = \left(\begin{array}{cc} \alpha & \gamma \\ \beta & \delta \end{array}\right) =\alpha \delta - \beta \gamma.
\end{equation}

If a sublattice is obtained from the linear lifting of an element $u$ of $\WT{\mathcal{M}}_{0,4,d} (1)$, its index has to agree with the degree of $u$.

\begin{lemma}\label{lem:degindcomp}
The sublattice $\Lambda_u$ for a degree $d$ holomorphic orbi-sphere $u: \PB(\Lambda) \to \PB(\Lambda_0)$ has index $d$.
\end{lemma}
\begin{proof}
The linear lifting $\WT{u}$ induces a map $\bar{u}:E_{\Lambda} \to E_{\Lambda_0}$ whose kernel is $\WT{u}^{-1} (\Lambda_0) / \Lambda$. Therefore
$$ \deg u = \deg \bar{u} = | \WT{u}^{-1} (\Lambda_0) / \Lambda | = | \Lambda_0 / \WT{u} (\Lambda)| = \ind (\Lambda_u).$$
\end{proof}

We remark that sublattices $\Lambda_u$ and $\Lambda_{u'}$ for two different elements $u$ and $u'$ in $\mathcal{M}$ can coincide. Our next task is to resolve this ambiguity.

\subsection{The moduli of holomorphic orbi-spheres and the set of sublattices}\label{subsec:holoorbsublat}

Consider two non-constant holomorphic orbi-spheres  $u : \PB(\Lambda) \to  \PB(\Lambda_0)$ and $u' : \PB(\Lambda') \to \PB(\Lambda_0)$ in $\WT{\mathcal{M}}$, and write $z_i$, $z_i'$ for (orbi-)marked point in $\PB(\Lambda)$, $\PB(\Lambda')$ respectively. By the definition of $\WT{\mathcal{M}}_{0,4,d}(1)$ \eqref{eqn:tempM}, we have $u(z_1) = u'(z'_1) = x_1$.
$u$ and $u'$ induce the linear liftings $\WT{u}(z) = \xi z$ and $\WT{u'}(z) = \xi' z$ together with sublattices $\Lambda_u = \xi\cdot \Lambda$ and $\Lambda_{u'} = \xi \cdot \Lambda'$ by Proposition \ref{prop:lifting}. Then we have the following lemma.

\begin{lemma}\label{Lem:equiv_rel}
$\Lambda_u$ and $ \Lambda_{u'}$ are the same sublattice in $\Lambda_0$ if and only if there is a biholomorphism of orbifolds, $\psi : \PB(\Lambda) \to \PB(\Lambda')$, between domains of $u$ and $u'$ such that $\psi (z_1) = z_1'$ and $u= u' \circ \psi$.
\end{lemma}

Readers are warned that $\psi$ does not give an equivalence between $u$ and $u'$ in the sense of Definition \ref{def:Mifjkl} in general, since $\psi$ may not preserve the other three orbi-marked points. Since $\psi$ is an orbifold isomorphism, at least we have $\psi (z_j) = z_{\tau(j)}'$ for a permutation $\tau$ on $\{1,2,3,4\}$ with $\tau(1)=1$.

\begin{proof}
First suppose that $\Lambda_u=\xi \Lambda$ and $\Lambda_{u'}=\xi \Lambda'$ are the same sublattice in $\Lambda_0$.
We set $\WT{\psi}$ to be the linear map $z \mapsto (\xi')^{-1}\xi z$ on $\C$, which obviously induces a map $\psi: \PB(\Lambda) \to \PB(\Lambda')$. Since $\WT{u'} \circ \WT{\psi}  = \WT{u}$, we see that the induced map $u= u' \circ \psi$. $\psi(z_1) = z_1'$ is automatic, since every maps we consider here preserve the base points.

Conversely, assume that there is a biholomorphism of orbifolds $\psi$ with $\psi(z_1)= z_1'$ and $u = u' \circ \psi$. We take the lift $\psi$ to the level of universal covering to get $\WT{\psi}:\C \to \C$, which is a linear map by the same argument as in the proof of Proposition \ref{prop:lifting}. Moreover, it is easy to see that $\WT{\psi} (\Lambda) = \Lambda'$.
\begin{equation}\label{dia:prism}
\xymatrix{
{} && {} & (\C,0) \ar[dd]^{p} \\
(\C,0) \ar[rr]_{\WT{\psi}} \ar[rrru]^{\WT{u}} \ar[dd]_{\pi} && (\C,0) \ar[ru]_{\WT{u'}} \ar@<0.7ex>[dd]_{\pi'} & {} \\
{} && {} & (\PB(\Lambda_0),x_1) \\
(\PB(\Lambda), z_1) \ar[rr]_{\psi} \ar[rrru]_{u} && (\PB(\Lambda'),z_1') \ar[ru]_{u'} & {}\\
}
\end{equation}

Consider two maps $\WT{u}$ and $\WT{u'} \circ \WT{\psi}$. Both of maps are the lifting of the same map $u  \circ \pi (= u' \circ \psi \circ \pi)$ and preserve the base point $0$. Thus they must differ by the deck transformation for $\C (\stackrel{p}{\longrightarrow} \PB(\Lambda_0))$ which fixes the origin. Therefore, we have $\WT{u}  = \pm \WT{u'} \circ \WT{\psi}$. (i.e. the upper triangle in the diagram \eqref{dia:prism} commutes up to sign.) Finally,
$$ \Lambda_u = \WT{u}(\Lambda) = \pm \WT{u'} ( \WT{\psi} (\Lambda) ) = \pm \WT{u'} (\Lambda') = \Lambda_{u'},$$
which finishes the proof.
\end{proof}

In particular, if $u'  \sim u$ (see Definition \ref{def:Mifjkl}), then $\Lambda_{u'} = \Lambda_u$. Recall that $\mathcal{M}_{0,4,d} (\D{1} , \D{j}, \D{k}, \D{l})$ is the quotient of $\WT{\mathcal{M}}_{0,4,d} (\D{1} , \D{j}, \D{k}, \D{l})$ by $\sim$ and
$\mathcal{M}_{0,4,d}(1) =\bigcup_{j,k,l} \mathcal{M}_{0,4} (\D{1} , \D{j}, \D{k}, \D{l})$. Hence, Proposition \ref{prop:lifting} and Lemma \ref{lem:degindcomp} gives us the correspondence
\begin{equation}\label{eqn:orbsp_lattice}
\Theta_d: \mathcal{M}_{0,4,d}(1) \to \mathcal{L}_d (\Lambda_0)=\{ \Lambda \subset \Lambda_0 : \Lambda \cong \Z^2,\,\ind (\Lambda)= d\}, \quad [u] \mapsto \Lambda_u
\end{equation}
which assigns a lattice to each (class of) holomorphic orbi-sphere that contributes to $\langle \D{1} , \D{j}, \D{k}, \D{l} \rangle_{0,4,d}$. However, $\Theta$ can not be $1$-to-$1$ since the same underlying map $u$ give rise to several distinct holomorphic orbi-spheres depending on the positions of their domain marked points $z_2,z_3,z_4$ (c.f. Lemma \ref{Lem:equiv_rel}). Indeed, we have the following.

\begin{prop}\label{thm:sublattice_corr}
The correspondence $\Theta_d$ \eqref{eqn:orbsp_lattice} is a surjective $6$-to-$1$ map, and each pair of classes $[u]$ and $[u']$ in $\Theta_d^{-1} (\Lambda)$ is related by a domain reparameterization $\psi$ that fixes the first marking $z_1$, but permutes other three; $z_2,z_3,z_4$.
\end{prop}
\noindent(Thus the appearance of $6$ in the statement is due to $|S_3|= 3! = 6$.)

\vspace{0.2cm}

We first prove the following simple lemma.
\begin{lemma}\label{lem:idlift}
For a given sublattice $\Lambda$ of $\Lambda_0$, there is a holomorphic orbi-sphere $u$ whose domain is $\PB(\Lambda)$ such that $ \Lambda_u  = \Lambda$. i.e. the linear lifting of $u$ is the identity. (In particular, $\Theta$ is surjective.)
\end{lemma}

\begin{proof}
Let us consider the identity map on $\C$, $id: \C \to \C$. Since $id$ sends $\Lambda$ into $\Lambda_0$ and $Z_2$-equivariant, it induces a holomorphic map
$$ u: \PB(\Lambda) = [E_{\Lambda}/ \Z_2] \to \PB(\Lambda_0) = [E_{\Lambda_0}/ \Z_2]$$
Since the image of four orbifold points in $\PB(\Lambda)$ lies in the set of orbifold points in $\PB(\Lambda_0)$ and the local behavior of $u$ around each orbifold point is simply $id$ on the cover, it is clear that $u$ is an element of $\mathcal{M}_{0,4,d}(1)$. Moreover, since the linear lifting of $u$ is the identity map on $\C$, $\Lambda_u = id (\Lambda) = \Lambda$, which proves the lemma.
\end{proof}

Now we are ready to prove the proposition.

\begin{proof}[Proof of Proposition \ref{thm:sublattice_corr}]

For any sublattice $\Lambda$, we already have $u :\PB(\Lambda) \to \PB(\Lambda_0)$ whose linear lifting is the identity map so that $\Lambda_u = \Lambda$ by Lemma \ref{lem:idlift}. In particular, $\Theta_d$ is a surjective map.  However, such $u$ produces more than one  elements in $\mathcal{M}_{0,4,d} (1)$ since we have several choices for the locations of $z_2,z_3,z_4$ in the domain $\PB(\Lambda)$ (see Remark \ref{rmk:234s3} for related discussion). Clearly, there are $3!=6$ such choices (parametrized by $S_3$; two examples of such configurations are drawn in Figure \ref{fig:sublat_markconfig}).

It is elementary to check that these six configurations define mutually different elements in $\Theta_d^{-1} (\Lambda)$. In fact, if there is an equivalence between two of them, $\Lambda$ should admit a symmetry represented by $z \mapsto \zeta z$ (for $\zeta \neq \pm 1$). Among sublattices of $\Lambda_0$, only $ \Lambda_0$ and its scaling can have such symmetry with $\zeta = \sqrt{-1}$, which obviously can not preserve the positions of markings $z_2,z_3,z_4$.

It remains to show that any element in $\Theta_d^{-1} (\Lambda)$ is equivalent to one of these six holomorphic orbi-spheres represented by $u$ whose linear lifting is the identity.
For $u' : \PB(\Lambda') \to \PB(\Lambda_0)$ with $\Lambda_{u'} = \Lambda$, take the linear lifting $z \mapsto \xi' z$ of $u'$ on $\C$.
\begin{equation*}
\xymatrix{ \C \ar[r]^{z \mapsto \xi'^{-1} z} \ar[d]_{\pi}& \C \ar[r]^{z \mapsto \xi' z} \ar[d]^{\pi'} & \C \ar[d]^{p} \\
\PB(\Lambda) \ar[r]^{\phi} &\PB(\Lambda') \ar[r]^{u'}& \PB(\Lambda_0) \\
}
\end{equation*}
Define $\phi : \PB(\Lambda) \to \PB(\Lambda')$ to be the biholomorphism induced by the linear map $z \mapsto \xi'^{-1}$ as in the above diagram. Here, the domain of $\phi$ should be $\PB(\Lambda)$ since $\Lambda = \Lambda_{u'} (= \xi' \Lambda')$ implies $\xi'^{-1} \Lambda = \Lambda'$. If we set the marked points in the domain of $u' \circ \phi$ to be $\phi^{-1}$ of those in the domain of $u'$, we see that $u' \circ \phi \sim u'$ (i.e. $[u' \circ \phi] = [u']$ in $\mathcal{M}_{0,4,d}(1)$) and the linear lifting of $u' \circ \phi$ is identity by the construction.
\end{proof}

\begin{figure}[h]
\begin{center}
\includegraphics[height=1.6in]{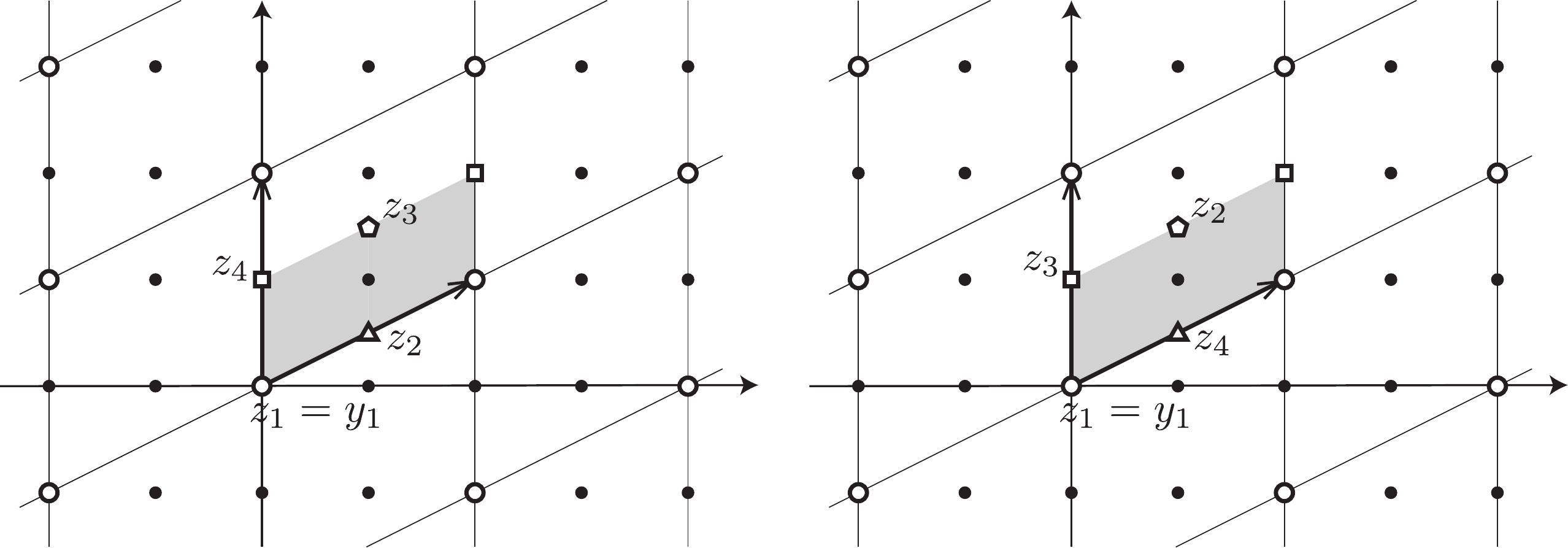}
\caption{$(u,\tau)$ for $\tau =id$ and $\tau = (2,4,3)$ }\label{fig:sublat_markconfig}
\end{center}
\end{figure}

The above proof provides a  concrete description of the fiber of $\Theta_d$. From the proof of the theorem, we see that an element in $\Theta_{d}^{-1} (\Lambda)$ can be represented by the projection of the identity as in the diagram below:
\begin{equation*}
\xymatrix{ \C \ar[r]^{id} \ar[d]_{\pi} &  \C \ar[d]^{p} \\
\PB(\Lambda) \ar[r]^{u} & \PB(\Lambda_0)}
\end{equation*}
Recall from \eqref{eqn:orbyi} that $y_1 = 0$, $y_2 = \frac{v_1}{2}$, $y_3= \frac{v_1 + v_2}{2}$, $y_4=\frac{v_2}{2}$ (modulo $\Lambda$)
for a chosen $\Z$-basis $\{v_1,v_2\}$ of $\Lambda$. (The choice of a basis will be fixed in \ref{subsec:clsublat}.)
As we pointed out in the proof of Proposition \ref{thm:sublattice_corr}, $u$ can represent six different classes in $\mathcal{M}_{0,4,d} (1)$ according to the position of marked points, and we have the following choices of $\{z_i\}$:
\begin{equation*}\label{eqn:markedarrange}
(z_1,z_2,z_3,z_4) = (y_{\tau(1)},y_{\tau(2)},y_{\tau(3)},y_{\tau(4)})
\end{equation*}
for a permutation $\tau$ with $\tau(1)=1$ (see Figure \ref{fig:sublat_markconfig}). Here $z_1=y_1$ due to Assumption \ref{assm:firstmark}. Thus $\tau$ can be thought of as an element of $S_3$.

From now on, we will always choose a representative $u$ of a class in $\mathcal{M}_{0,4,d} (1)$ such that the linear lifting of $u$ is identity.

\subsection{Classification of sublattices}\label{subsec:clsublat}
Proposition \ref{thm:sublattice_corr} reduces our counting problem essentially to a classification of 2-dimensional sublattices of $\Lambda_0$. In this section we investigate the counting problem of sublattices $\Lambda$ with $\ind (\Lambda)=d$, that is, finding the number of elements in $\mathcal{L}_d (\Lambda_0)$ ((3) of Definition \ref{def:indindlat}).
For this, we first choose a special $\Z$-basis for a sublattice $\Lambda$.
\begin{lemma}\label{lem:basisfix}
For any sublattice $\Lambda \subset \Lambda_0$, there is a $\Z$-basis $\{ w_1, w_2 \}$ of $\Lambda$ such that $w_1 = (h,0)$ for a positive divisor $h$ of $\ind ( \Lambda)$.
\end{lemma}
\begin{proof}
Choose any $\Z$-basis $\{ v_1, v_2 \}$ of $\Lambda$ which is positively oriented as basis of $\R^2 \cong \C$. We may write $v_1 = (\alpha,\beta)$ and $v_2 = (\gamma,\delta)$ for some $\alpha,\beta,\gamma,\delta \in \Z$ with respect to the standard $\Z$-basis $\{ 1,\sqrt{-1} \}$ of $\Lambda_0 \subset \C$. Then we have $d:=\alpha \delta - \beta \gamma>0$. Recall from \eqref{eqn:degdet} that the index of $\Lambda$ is given by $d$ in this case.

Let $g := \gcd(\beta,\delta) \in \N$, and put $\beta=g\beta'$, $\delta=g\delta'$ with $\gcd(\beta',\delta') = 1$. Hence, one can find integers $x$ and $y$ satisfying
$$\beta' x + \delta' y = 1.$$
Now, we set the new basis to be
$$ w_1 = \delta' v_1 -\beta' v_2 ,\quad w_2 = x v_1 + y v_2.$$
It is clearly a basis of $\Lambda$ since
\begin{equation*}
	\left(
		\begin{array}{c c}
			\delta' & x\\
			-\beta' & y
		\end{array}
	\right) \in SL_2 (\Z),
\end{equation*}
In terms of coordinates, $\{w_1,w_2\}$ can be written as
$$w_1 = (h,0), \quad w_2=( \alpha x +  \gamma y,g)$$
where $h$ satisfies $hg = d$ since $\det (w_1 \, w_2) = \det(v_1 \,v_2)$. Therefore $h$ is a positive divisor of $d$ (since $g>0$).
\end{proof}

Now we are ready to compute the number of elements in $\mathcal{L}_d (\Lambda_0)$, which will give rise to the number of certain holomorphic orbi-spheres in $\PB$ due to Proposition \ref{thm:sublattice_corr}.

\begin{lemma}
The set $\mathcal{L}_{d} (\Lambda_0)$ of index $d$ sublattices of $\Lambda_0$ consists of $\mathfrak{D} (d)$ elements, where $\mathfrak{D}$ is the divisor sum function. (i.e. $\mathfrak{D} (n) := \sum_{k|n} k$ for $n \in \Z_{>0}$.)
\end{lemma}
\begin{proof}
Observe that $w_1$ is a primitive vector in $\Lambda$ in order to form a $\Z$-basis of $\Lambda$.
Hence, it is uniquely determined by the property that it lies along the positive real axis. Thus one can fix the first basis element by $w_1=(h,0)$ without ambiguity for a given $\Lambda \in \mathcal{L}_{d} (\Lambda_0)$. 

Since the index of $\Lambda$ is $d$, it gives some restriction on the second basis element, namely, $ w_2=(m,d/h)$ for some $m \in \Z$ (recall $h | d$ from the previous lemma). It is easy to see that the following $h$-tuples (obtained from $0 \leq m \leq h-1$) generate mutually different sublattices of $\Lambda_0$
$$\{(h,0), (0,d/h)\},\{(h,0), (1,d/h)\}, \cdots,\{(h,0), (h-1,d/h)\}.$$
(See Figure \ref{fig:basisfix} for examples.)
However, $\{(h,0), (0,d/h)\}$ and $\{(h,0), (h,d/h)\}$ generate the same sublattice since they are related by
\begin{equation*}
	T = \left(
		\begin{array}{c c}
			1 & 1\\
			0 & 1
		\end{array}
	\right) \in SL_2 (\Z).
\end{equation*}

In conclusion, there are exactly $h$ sublattices in $\mathcal{L}_d (\Lambda_0)$ with $w_1 = (h,0)$. As $h$ can vary over the set of positive divisors of $d$, we see that $|\mathcal{L}_d (\Lambda_0)| = \mathfrak{D} (d)$.
\end{proof}

In terms of basis of the sublattice, $|\mathcal{L}_d (\Lambda_0)|$ is the number of solutions of the determinant equation $\det\left(\begin{array}{cc}\alpha & \gamma \\\beta & \delta\end{array}\right) =d$ \eqref{eqn:mainadbcN1} where $ \left(\begin{array}{cc}\alpha & \gamma \\\beta & \delta\end{array}\right)$ is taken from $M_{2 \times 2} (\Z) / SL(2,\Z)$. We have just seen that the number of such solutions is exactly $\mathfrak{D}(d)$.

As a direct corollary from Proposition \ref{thm:sublattice_corr}, we obtain the following formula which recovers nonconstant terms of $f(q)$ in \eqref{eqn:nonconstf}.

\begin{theorem}\label{thm:lumpsum} The power series $\sum_{d \geq 1} |\mathcal{M}_{0,4,d} (1)| q^d$ equals to
\begin{equation}\label{eqn:degreelatticeseries}
 6 \sum_{d=1} | \mathcal{L}_d (\Lambda_0)| \, q^d  = 6 \sum_{d=1} \mathfrak{D}(d) q^d = 6 f(q).
 \end{equation}
 \end{theorem}
We will now refine  this ``combined counting" \eqref{eqn:degreelatticeseries} further in accordance with the images of three orbifold points $y_2,y_3,y_4$ in the domain of holomorphic orbi-sphere (recall $u(y_1)=x_1$ for $[u] \in \mathcal{M}_{0,4,d} (1)$).

\begin{figure}[h]
\begin{center}
\includegraphics[height=1.7in]{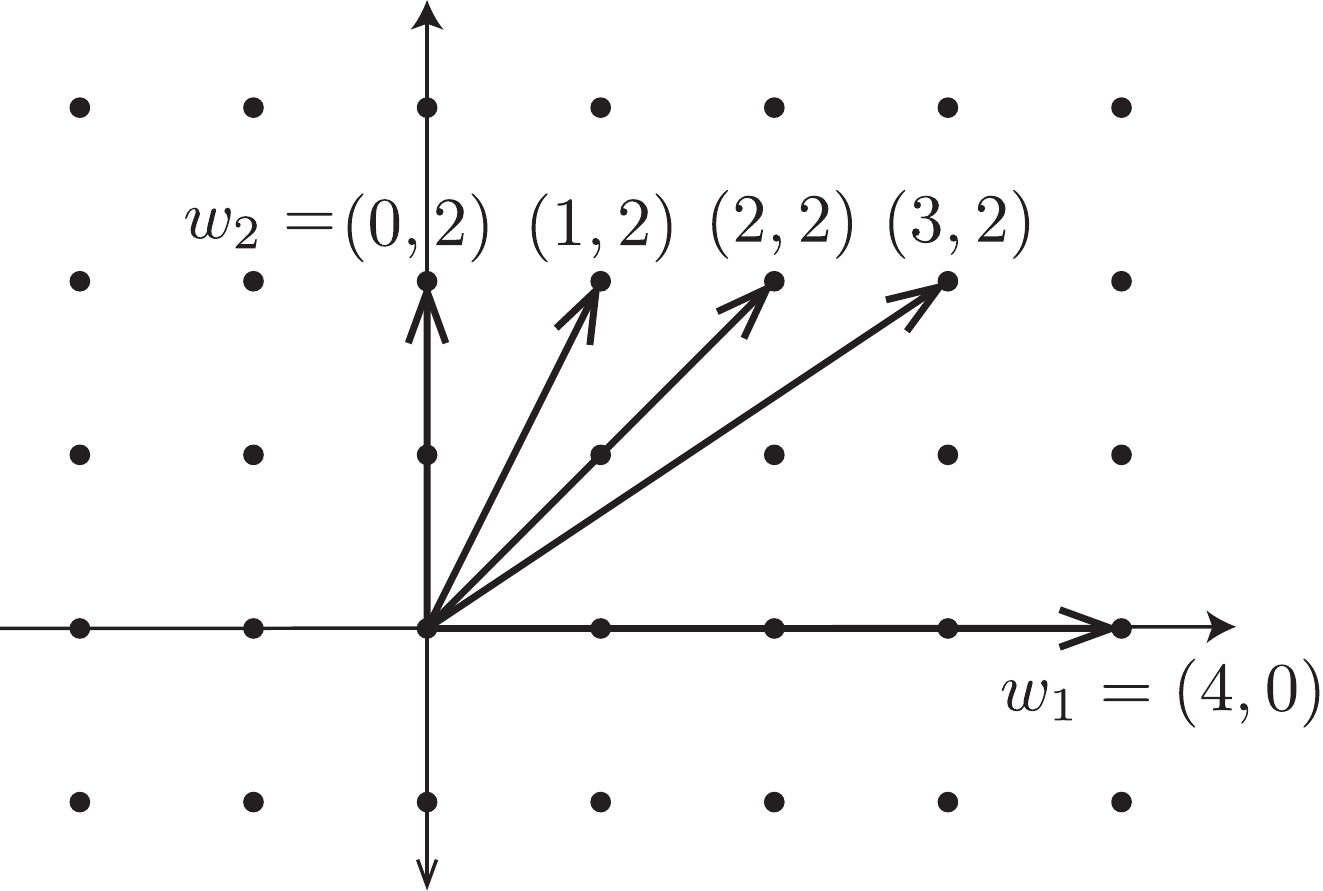}
\caption{$\Lambda_{4,2}^m$ for $m=0,1,2,3$ ($d=4\cdot2=8$)}\label{fig:basisfix}
\end{center}
\end{figure}

For $hg =d$, we define $\Lambda^m_{h,g}$ by the index $d$ sublattice generated by $\{(h,0),(m,g)\}$ for $0 \leq m \leq h-1$ so that
$$\mathcal{L}_d (\Lambda_0) = \{ \Lambda^m_{h,g} \,|\, hg=d, 0 \leq m \leq h-1\}.$$
(From the proof of Lemma \ref{lem:basisfix}, we know that $\Lambda^m_{h,g} = \Lambda^{m+h}_{h,g}$.)
Recall that we have defined the $6$-to-$1$ correspondence
$$\Theta_d : \mathcal{M}_{0,4,d} (1) =\bigcup_{i,j,k} \mathcal{M}_{0,4,d} (\D{1}, \D{j}, \D{k}, \D{l}) \to \mathcal{L}_d (\Lambda_0).$$

We now look into the inverse image $\Theta_d^{-1} (\Lambda^m_{h,g})$.
From the discussion at the end of \ref{subsec:holoorbsublat}, the class in $\Theta_d^{-1} (\Lambda^m_{h,g})$ admits a representative $u: \PB(\Lambda^m_{h,g} ) \to \PB(\Lambda_0)$ whose linear lifting is identity. As in \eqref{eqn:orbyi}, $y_1,y_2,y_3,y_4 \in \PB(\Lambda)$ will denote four orbifold points in the domain of $u$:
$$y_1=0,\quad y_2=\frac{w_1}{2},\quad y_3=\frac{w_1 + w_2}{2}, \quad y_4=\frac{w_2}{2} \mod \Lambda^m_{h,g},$$
but now, with respect to the fixed basis $w_1 = (h,0)$ and $w_2=(m,g)$ of $\Lambda^m_{h,g}$.

\begin{prop}\label{prop:insert_type}
For each sublattice $\Lambda^m_{h,g}$, choose a representative $u$ of $[u] \in \Theta_d^{-1} (\Lambda^m_{h,g})$ whose linear lifting is identity. Then one can determine the orbi-insertions of $[u] \in \Theta_d^{-1} (\Lambda^m_{h,g})$ at $(y_2,y_3,y_4)$ as follows:
($u(y_1)$ is always $x_1$)
\begin{itemize}
\item if $d$ is even,
	\begin{enumerate}
    \item[(i)] $(g, h, m)\equiv (0,0,0) \mod 2\Rightarrow$ $u(y_2,y_3,y_4)=(x_1,x_1,x_1)$;
    \item[(ii)] $(g, h, m)\equiv (0,0,1) \mod 2\Rightarrow$ $u(y_2,y_3,y_4)=(x_1,x_2,x_2)$;
    \item[(iii)] $(g, h, m)\equiv (0,1,0) \mod 2\Rightarrow$ $u(y_2,y_3,y_4)=(x_2,x_2,x_1)$;
    \item[(iv)] $(g, h, m)\equiv (0,1,1) \mod 2\Rightarrow$ $u(y_2,y_3,y_4)=(x_2,x_1,x_2)$;
    \item[(v)] $(g, h, m)\equiv (1,0,0) \mod 2\Rightarrow$ $u(y_2,y_3,y_4)=(x_1,x_4,x_4)$;
    \item[(vi)] $(g, h, m)\equiv (1,0,1) \mod 2\Rightarrow$ $u(y_2,y_3,y_4)=(x_1,x_3,x_3)$;
    \end{enumerate}
\item if $d$ is odd,
	\begin{enumerate}
	\item[(vii)] $(g, h, m)\equiv (1,1,0) \mod 2\Rightarrow$ $u(y_2,y_3,y_4)=(x_2,x_3,x_4)$;
	\item[(viii)] $(g, h, m)\equiv (1,1,1) \mod 2\Rightarrow$ $u(y_2,y_3,y_4)=(x_2,x_4,x_3)$.
	\end{enumerate}
\end{itemize}
\end{prop}

\begin{proof}
The proof follows directly from the picture. For example, if $g,h,m$ are all even, all coordinates of $y_2,y_3,y_4$ become integers since
\begin{equation}\label{y234ghm}
y_2 = \left(\frac{h}{2},0 \right),\quad y_3 = \left(\frac{h+m}{2},\frac{g}{2} \right),\quad y_4= \left(\frac{m}{2},\frac{g}{2} \right) \mod \Lambda_{g,h}^m.
\end{equation}
This implies that $u(y_i) = 0 \mod \Lambda_0$ and the case (i) follows as $x_1=0 \mod \Lambda_0$ \eqref{eqn:arrangementxi}. (Here, we used the fact that the linear lifting of $u$ is the identity map.)

The rest of cases can be deduced similarly by comparing \eqref{y234ghm} and \eqref{eqn:arrangementxi}.
\end{proof}

To compute corresponding Gromov-Witten invariants precisely, we need to additionally take into account the fact that six elements in $\Theta_d^{-1} (\Lambda^m_{h,g})$ are assigned to six different configuration of three marked points $(z_2,z_3,z_4)$. Namely, $(y_2,y_3,y_4) = (z_{\tau(2)}, z_{\tau(3)},z_{\tau(4)})$ for a permutation $\tau$ on $\{2,3,4\}$. This will be dealt with in the next section.

\section{Computation of genus 0 Gromov-Witten invariant of $\PB$}\label{Sec:computation}
 In this section, we precisely compute the genus 0 GW potential of $\PB$ using the counting result obtained in Section \ref{sec:mainsec_sphcount}. At the end, we will see that our computation matches the formula of the GW potential given in \cite{ST}. As in the previous section, we fix the complex structure of the target orbifold of the holomorphic orbi-spheres, and write it as $\PB(\Lambda_0)$.

Since all cone points in $\PB (\Lambda_0)$ have $\Z_2$-singularities, the domain orbifold  with only $\Z_2$-singularities can contribute to the Gromov-Witten invariant.
Since the divisor axiom will determine the invariants with $\pt$-insertions from those without $\pt$-insertions (except $\langle \mathbf{1}, \mathbf{1}, \pt \rangle$; see the discussion below), it is enough to consider the correlators containing $\mathbf{1}$ and $\D{i}$'s only. Among these correlators, only the following type can possibly survive (by the dimension formula (see Appendix \ref{app:orbcovth})):
\begin{equation}\label{eqn:possiblenonzerogw}
\langle \D{i}, \D{j}, \D{k}, \D{l}\rangle_{0,4,d}, \quad \langle \mathbf{1} , \D{i}, \D{j} \rangle_{0,3,d}.
\end{equation}

All possible domains for the nonconstant stable maps contributing to \eqref{eqn:possiblenonzerogw} are listed as follows:
\begin{enumerate}
\item[(a)] $\PP^1_{2,2}$ contributing to $\langle \mathbf{1}, \Delta_{\bullet}, \Delta_{\circ} \rangle$,
\item[(b)] $\PP^1_{2,2,2,2}$ contributing to $\langle \Delta_{\bullet}, \Delta_{\circ}, \Delta_{\diamond}, \Delta_{\dagger} \rangle$,
\item[(c)] (nodal domains) suitable connected sums of $\PP^1$,  $\PP^1_{2}$, $ \PP^1_{2,2}$, $\PP^1_{2,2,2}$.
\end{enumerate}

By the same argument as in the proof of Lemma \ref{Lem:degenerate}, one can show that $\langle \mathbf{1}, \mathbf{1}, \pt, \cdots, \pt \rangle$ and $\langle \mathbf{1}, \Delta_{\bullet}, \Delta_{\circ} \rangle$ are contributed by only constant maps.

By the divisor axiom, $\langle \mathbf{1}, \mathbf{1}, \overbrace{\pt, \cdots, \pt}^{n} \rangle$ is zero unless $n=1$. When $n=1$, $\langle \mathbf{1}, \mathbf{1}, \pt\rangle$ is nothing but the structure constant for the Chen-Ruan cup product since it only admits constant map contribution. As $\mathbf{1}$ and $\pt$ are from the trivial twisted sector, the corresponding product is simply the cup product for $S^2$, which is $\mathbf{1} \cup \mathbf{1} = PD(\pt) = \mathbf{1}$. Hence the structure constant is $1$, or equivalently, $\langle \mathbf{1}, \mathbf{1}, \pt\rangle=1$.

Likewise, $\langle \mathbf{1}, \Delta_{\bullet}, \Delta_{\circ} \rangle$ in (a) is the coefficient of $\pt$ for the Chen-Ruan cup product $\Delta_{\bullet} \cup \Delta_{\circ}$, which we will examine in \ref{subsec:compst1}. Nonconstant contributions for (b) in the above list will be computed in \ref{subsec:idorins}.

We first exclude the possibility of contribution from (c). (Recall from \eqref{eqn:tempM} that the counting in the previous section does not include such maps.)

\subsection{Nodal curve contributions}\label{subsec:nodalcontri}
To show that (c) in the above list never happens in our case, we prove that there are no holomorphic maps from each of $\PP^1$,  $\PP^1_{2}$, $ \PP^1_{2,2}$, $\PP^1_{2,2,2}$ to $\PB$.

\begin{lemma}\label{Lem:degenerate}
There is no nodal curves contributing to the genus-0 Gromov-Witten invariants except degree zero maps.
\end{lemma}
\begin{proof}
We divide possible irreducible components of nodal maps into two cases as follows.

\noindent(i) First, $\PP^1$, $\PP^1_{2,2,2}$ and $\PP^1_{2,2}$ are given as the quotients of $\PP^1$ by finite groups. Since $\PP^1$ is simply connected, any non-constant holomorphic map from these domains to $\PB$ can be lifted to a non-constant holomorphic map from $\PP^1$ to the elliptic curve $E_0$ (recall $\PB(\Lambda_0) = [E_0 / \Z_2]$.
Since homotopy types such maps should be trivial in $\pi_2 (E_0)$, they should be constant maps.

\noindent(ii)
For $\PP^1_2$, we use the fact that the homomorphism $\pi_1^{orb} ([V / G]) \to \pi_1^{orb} (\CX)$ induced from any local chart $[V/G]$ of $\CX$ is injective for a good orbifold $\CX$. We apply this to $\CX = \PB(\Lambda_0)$.
If there is any non-constant morphism $f$ from $\PP^1_2$ to $\PB(\Lambda_0)$,  it should look locally around the singular point of $\PP^1_2$ as follows:
\begin{equation*}
\xymatrix{ \PP^1_2 \ar[r]^{f} & \PB(\Lambda_0) \\
[U / \Z_2] \ar[r]^{f|_U} \ar[u] &  [V / \Z_2] \ar[u] }
\end{equation*}
for some discs $U, V \subset \C$. This induces homomorphisms between corresponding orbifold fundamental groups
\begin{equation*}
\xymatrix{
\pi_1^{orb} (\PP^1_2) \ar[r] & \pi_1^{orb} (\PB(\Lambda_0)) \\
\pi_1^{orb} ([U / \Z_2]) (\cong \Z_2) \ar[r]^{\text{inj}} \ar[u]_0 & \pi_1^{orb} ([V / \Z_2])(\cong \Z_2) \ar[u]^{\text{inj}} }.
\end{equation*}
where the vertical map on the right is injective by the property of a good orbifold above, and the bottom map is injective by \eqref{eqn:injpi1}. However,
the left vertical map is a zero map because $\pi_1^{orb} (\PP^1_2) = 0$, so the above diagram can not be commutative.
\end{proof}


\subsection{Identification of orbi-insertions}\label{subsec:idorins}
We now give a precise computation of the Gromov-Witten invariants for $\PB$ of the type $\langle \Delta_{\bullet}, \Delta_{\circ}, \Delta_{\diamond}, \Delta_{\dagger} \rangle$ with nonconstant contributions. So, the degree $d$ will be a positive integer throughout the section.

By Proposition \ref{thm:sublattice_corr}, we associate six non-equivariant orbifold stable maps for each sublattice $\Lambda$. More concretely, for $\Lambda^m_{g,h} \subset \Lambda_0$ we have six maps $(u,\tau)$ for $\tau \in S_3$ where
$$ u : \PB(\Lambda^m_{g,h}) \to \PB(\Lambda_0)$$
with the identity as the linear lifting and the configuration of the marked points is given as
\begin{equation}\label{eqn:tautautau}
(z_2, z_3, z_4) = (y_{\tau(2)},y_{\tau(3)},y_{\tau(4)}).
\end{equation}

Below, we will use following formal power series to express coefficients of the Gromov-Witten potential of $\PB$:
\begin{align*}
\mathfrak{D}(q) :=& \sum_{d \in \N} \mathfrak{D}(d) q^d\\
\mathfrak{D}^{odd}(q) :=& \sum_{d \in \N \setminus 2\N} \mathfrak{D}(d) q^d\\
\mathfrak{D}^{even}(q) :=& \sum_{d \in 2\N} \mathfrak{D}(d) q^d.
\end{align*}

\begin{prop}\label{prop:nonconstgw}
Counting nonconstant holomorphic orbi-spheres in $\PB$ gives the following genus 0 Gromov-Witten invariants.
\begin{equation*}
\begin{array}{l}
 \displaystyle\sum_{d \geq 1} \langle \D{1}, \D{2}, \D{3}, \D{4} \rangle_{0,4,d} \, q^d = \mathfrak{D}^{odd}(q)\\
\displaystyle\sum_{d \geq 1} \langle \D{i}, \D{i}, \D{i}, \D{i} \rangle_{0,4,d} \, q^d = 6 \mathfrak{D}(q^4) \\
\displaystyle\sum_{d \geq 1}  \langle \D{i}, \D{i}, \D{j}, \D{j} \rangle_{0,4,d} \, q^d =  \frac{2}{3} \left(\mathfrak{D}^{even}(q) - \mathfrak{D}(q^4) \right)
\end{array}
\end{equation*}
(for $i \neq j$ in the last equation).
\end{prop}

\begin{proof}
\noindent (i) $\langle \D{1}, \D{2}, \D{3}, \D{4} \rangle_{0,4,d}$:
Since the image of $\{y_1,y_2,y_3,y_4\}$ under $u$ should be $\{x_1,x_2,x_3,x_4\}$, we need to consider the maps $u$ with $\Theta_d ([u]) = \Lambda_{g,h}^m$ for
\begin{itemize}
\item
$(g,h,m) \equiv (1,1,0) \mod 2$ ($\stackrel{\rm Prop. \ref{prop:insert_type}}{\Longrightarrow} u(y_2,y_3,y_4) = (x_2,x_3,x_4)$), or
\item$(g,h,m) \equiv (1,1,1)  \mod 2$ ($\stackrel{\rm Prop. \ref{prop:insert_type}}{\Longrightarrow} u(y_2,y_3,y_4)  =(x_2, x_4,x_3)$).
\end{itemize}
In particular, $d=gh$ is odd in his case.

For $\langle \D{1}, \D{2}, \D{3}, \D{4} \rangle_{0,4,d}$, we have to count the map whose marked points $z_2,z_3,z_4$ map precisely to $x_2,x_3,x_4$ respectively. Hence, for $(g,h,m) \equiv (1,1,0)$, we should choose $\tau$ in \eqref{eqn:tautautau} to be identity so that $(y_2, y_3, y_4) = (z_2,z_3,z_4)$. Similarly, $\tau$ should be the transition $(3,4)$ for $(g,h,m) \equiv (1,1,1)$.

Consequently, for any odd  $d$ together with the factorization $d=gh$, we have $h$ (equivalnce class of) holomorphic orbi-spheres contributing to $ \langle \D{1}, \D{2}, \D{3}, \D{4}\rangle_{0,4,d}$ (as $m$ ranges over $0,1, \cdots, h-1$), i.e.
\begin{equation*}
 \langle \D{1}, \D{2}, \D{3}, \D{4}\rangle_{0,4,d} = \left\{
\begin{array}{ll}
\sum_{ h | d} h =\mathfrak{D} (d) & d \equiv 1 \mod 2 \\
0 & d \equiv 0 \mod 2
\end{array}\right.
\end{equation*}
and hence $\sum_{d \geq 1} \langle \D{1}, \D{2}, \D{3}, \D{4} \rangle_{0,4,d} \,q^d = \mathfrak{D}^{odd}(q)$.

\vspace{0.5cm}
\noindent (ii) $\langle \D{1}, \D{1}, \D{1}, \D{1} \rangle_{0,4,d}$: In this case, $u$ sends all orbi-points $y_1,y_2,y_3,y_4$ to $x_1$, and hence $\Theta_d ([u]) = \Lambda_{g,h}^m$ with
\begin{itemize}
\item
$(g,h,m) \equiv (0,0,0) \mod 2$ ($\stackrel{\rm Prop. \ref{prop:insert_type}}{\Longrightarrow} u(y_2,y_3,y_4) = (x_1,x_1,x_1)$)\end{itemize}
Moreover, $\tau$ in \eqref{eqn:tautautau} can be arbitrary.
Observe that $d=gh$ is a multiple of $4$. We set $d= 4d'$, $h=2h'$, $g=2g'$.

For any factorization $4d'=2g' \cdot 2h'$, we have $3! \times h'$ (equivalence classes of) holomorphic orbi-spheres contributing to $ \langle \D{1}, \D{1}, \D{1}, \D{1}\rangle_{0,4,d}$ as $\tau$ can be arbitrary elements in $S_3$ and $m$ ranges over $0,2,4, \cdots, 2h'-2$ (so there are $h'$ possiblities), i.e.
\begin{equation*}
 \langle \D{1}, \D{1}, \D{1}, \D{1}\rangle_{0,4,d} = \left\{
\begin{array}{ll}
\sum_{ h' | d'} 6h' = 6\mathfrak{D}(d') & d \equiv 0 \mod 4 \\
0 &  otherwise
\end{array}\right.
\end{equation*}
and hence
$$\sum_{d \geq 1} \langle \D{1}, \D{1}, \D{1}, \D{1} \rangle_{0,4,d} \,q^d = \sum_{d'} 6 \mathfrak{D}(d') q^{4d'} = 6 \mathfrak{D} (q^4).$$

The obvious symmetry on $\PB$ tells us that $\sum_{d \geq 1} \langle \D{i}, \D{i}, \D{i}, \D{i} \rangle_{0,4,d} \,q^d $ for $i=2,3,4$ admit the same expression as above.
\vspace{0.5cm}

\noindent (iii) $\langle \D{1}, \D{1}, \D{j}, \D{j} \rangle_{0,4,d}$ for $j = 2,3,4$:
Let us first consider the case of $\langle \D{1}, \D{1}, \D{4}, \D{4} \rangle$.
In this case, $\Lambda_{g,h}^m$ associated with $u$ should be of the type
\begin{itemize}
\item
$(g,h,m) \equiv (1,0,0) \mod 2$ ($\stackrel{\rm Prop. \ref{prop:insert_type}}{\Longrightarrow} u(y_2,y_3,y_4) = (x_1,x_4,x_4)$)\end{itemize}
In addition, we have to count the map whose marked points $z_2,z_3,z_4$ map to $(x_1,x_4,x_4)$ for $\langle \D{1}, \D{1}, \D{1}, \D{1} \rangle_{0,4,d}$. Hence, $\tau$ in \eqref{eqn:tautautau} should be either the identity or the transition $(3,4)$.

Note that $d=gh$ is even in this case. If we set $d= 2^{N} d'$ with $d'$ odd, then $h$ should be a multiple of $2^N$ so that $g$ becomes odd. Thus $h$ can be $2^N d''$ for any $d''|d'$, and after fixing $h=2^N d''$, the number of choices for $m$  is $2^N d'' / 2 = 2^{N-1} d''$ (since $m$ should be even). Therefore
\begin{equation}\label{eqn:1144middle}
\langle \D{1}, \D{1}, \D{4}, \D{4} \rangle_{0,4,d}= 2 \sum_{d'' | d'} 2^{N-1} d'' = 2^N \mathfrak{D} ( d')
\end{equation}
where $2$ in the front is due to two possible choices for $\tau$.

We next express \eqref{eqn:1144middle} only in terms of the degree $d$. Observe that
\begin{equation}\label{eqn:1144midaft}
\mathfrak{D} (d) = (1+ 2 + 2^2 +  \cdots + 2^N )\mathfrak{D} (d') = (2^{N+1} - 1) \mathfrak{D} (d')
\end{equation}
Hence, if $d \equiv 2 \mod 4$, then $N$ should be $1$ and $\mathfrak{D}(d) = 3 \mathfrak{D} (d')$ so that
$$\langle \D{1}, \D{1}, \D{4}, \D{4} \rangle_{0,4,d} =2  \mathfrak{D} ( d') = \frac{2}{3} \mathfrak{D} (d).$$
If $d \equiv 0 \mod 4$ (i.e. $N>1$), we use the identity $ \mathfrak{D} (d/4) = (2^{N-1} - 1) \mathfrak{D} (d')$. Combining with \eqref{eqn:1144midaft}, we get
$$\mathfrak{D} (d) - \mathfrak{D} (d/4 ) = 2^{N-1} 3 \mathfrak{D} (d').$$
Finally, if $d$ is a multiple of $4$, we have
$$\langle \D{1}, \D{1}, \D{4}, \D{4} \rangle_{0,4,d} =2^N \mathfrak{D} ( d') = \frac{2}{3} ( \mathfrak{D} (d) - \mathfrak{D}(d/4) ).$$

In summary,
\begin{equation*}
 \langle \D{1}, \D{1}, \D{4}, \D{4}\rangle_{0,4,d} = \left\{
\begin{array}{cl}
 \frac{2}{3} ( \mathfrak{D} (d) - \mathfrak{D}(d/4) ) & d \equiv 0 \mod 4 \\
\frac{2}{3} \mathfrak{D} (d) & d \equiv 2 \mod 4 \\
0 & d \equiv 1,3 \mod 4
\end{array}\right.
\end{equation*}
which implies
$$\sum_{d \geq 1} \langle \D{1}, \D{1}, \D{4}, \D{4} \rangle_{0,4,d} q^d =  \frac{2}{3} \left(\mathfrak{D}^{even}(q) - \mathfrak{D}(q^4) \right).$$

One can deduce
$$\sum_{d \geq 1} \langle \D{i}, \D{i}, \D{j}, \D{j} \rangle_{0,4,d} q^d =  \frac{2}{3} \left(\mathfrak{D}^{even}(q) - \mathfrak{D}(q^4) \right)$$
similarly for $i=1$ and $j=2,3$, and the identity holds for more general pair $(i,j)$ (with $i\neq j$) by the symmetry of $\PB$.
\end{proof}

\subsection{Comparison with the formula in \cite{ST}}\label{subsec:compst1}
Proposition \ref{prop:nonconstgw} completely determines the Gromov-Witten invariant of $\PB$ with nonconstant holomorphic orbi-sphere contributions. Before we deduce the GW potential from our counting, we briefly discuss the $(d=0)$-component of the Gromov-Witten invariant.

By degree reason (see \eqref{eqn:possiblenonzerogw}), we only have the following two possibilities
$$\langle 1, \D{j}, \D{j} \rangle_{0,3,0}, \quad \langle \D{j}, \D{j}, \D{j}, \D{j} \rangle_{0,4,0}.$$
The first term is nothing but structure constant of Chen-Ruan cup product, and is  given by
$$\langle 1, \D{j}, \D{j} \rangle_{0,3,0} =  \D{j} \cup \D{j} = \int_{[\{w_j\} / \Z_2]}^{orb} \D{j} \cup \iota^* \D{j} = \frac{1}{2}$$
where $\iota$ is the involution map on inertia orbifold $I\PB$, which is the identity map in our case.

There are also nontrivial contributions from constant maps to $\langle \D{j}, \D{j}, \D{j}, \D{j} \rangle_{0,4,0}$ for $1 \leq j \leq 4$. Indeed, these constant maps are obstructed, and hence, we need to analyze the obstruction bundles on the corresponding moduli spaces, which we will not go into details about (as it is a little far from the main theme of the paper). See Remark \ref{rmk:appconst1} for related discussion.

Combining our counting made in Proposition \ref{prop:insert_type} with these constant contributions, we obtain the genus 0 Gromov-Witten potential as follows:
\begin{equation*}\label{eqn:SHcomPB}
\begin{array}{rl}
F = & \frac{3}{3!} t_0^2 \log q + \frac{1}{3!}  \frac{3}{2} t_0 ( t_1^2 + t_2^2 + t_3^2 + t_4^2 ) \\ &+ \frac{1}{4!} 4!( t_1 t_2 t_3 t_4 ) \cdot \mathfrak{D}^{odd}(q)  + \frac{1}{4!} ( t_1^4 + t_2^4 + t_3^4 + t_4^4 ) \cdot  \left(6 \mathfrak{D} (q^4) -\frac{1}{4} \right)\\ &+ \frac{1}{4!} \frac{4!}{2! 2!} ( t_1^2 t_2^2 + t_1^2 t_3^2 + t_1^2 t_4^2 + t_2^2 t_3^2 + t_2^2 t_4^2 + t_3^2 t_4^2 ) \cdot   \frac{2}{3} \left(\mathfrak{D}^{even}(q) - \mathfrak{D}(q^4) \right) \\
=& \frac{1}{2} t_0^2 \log q + \frac{1}{4} t_0 ( t_1^2 + t_2^2 + t_3^2 + t_4^2 ) \\ &+ ( t_1 t_2 t_3 t_4 ) \cdot \mathfrak{D}^{odd}(q)  + \frac{1}{4} ( t_1^4 + t_2^4 + t_3^4 + t_4^4 ) \cdot  \left( \mathfrak{D} (q^4) -\frac{1}{24} \right)\\ &+ \frac{1}{6} ( t_1^2 t_2^2 + t_1^2 t_3^2 + t_1^2 t_4^2 + t_2^2 t_3^2 + t_2^2 t_4^2 + t_3^2 t_4^2 ) \cdot  \left(\mathfrak{D}^{even}(q) - \mathfrak{D}(q^4) \right).
\end{array}
\end{equation*}
Using the identity \eqref{eqn:nonconstf}, one can easily see that this agrees with the earlier computation of Satake-Takahashi \eqref{eqn:STcomPB}.

\appendix
\section{Regularity}\label{app:trans}
Here, we prove that all holomorphic orbi-spheres in $\WT{\mathcal{M}}_{0,4,d} (\D{i}, \D{j},\D{k},\D{l})$ (Definition \ref{def:Mifjkl}) have no obstructions, or equivalently, ${\rm \mathbb{E}xt}^2(u^*\Omega_{\CX}^1 \to \Omega_\mathcal{C}^1(\sum z_i), \CO_{\mathcal{C}})$ vanish for all $u \in \WT{\mathcal{M}}_{0,4,d} (\D{i}, \D{j},\D{k},\D{l})$.
This will justify that our direct counting of orbifold stable maps in Sections \ref{sec:mainsec_sphcount} and \ref{Sec:computation} genuinely computes Gromov-Witten invariants. In other words, this together with Lemma \eqref{Lem:degenerate} will prove Proposition \ref{prop:regnodeg}.

For a K\"ahler orbifold  $\CX$, consider the following diagram:
\begin{equation*}
\xymatrix{
\mathcal{U} \ar[r]^f \ar[d]_{\pi} & \CX\\
\OL\CM_{0,4,d}(\CX, \vec{\mathbf{x}})
}
\end{equation*}
where $\mathcal{U}$ is the universal family over $\OL\CM_{0,4,d}(\CX, \vec{\mathbf{x}})$ for a non-trivial element $d \in H_2 (\CX, \Q)$ and a quadruple $\vec{\mathbf{x}}$ of twisted sectors $[ \{ w_j \} / \Z_2]$.

From the tangent-obstruction exact sequence of sheaves on $\OL\CM_{0,4,d}(\CX, \vec{\mathbf{x}})$, we have the following exact sequence of vector spaces :
\begin{equation}\label{eq:tan_obs}
\begin{array}{rl}
0 & \to {\rm \mathbb{H}om}(f^*\Omega_{\CX}^1 \to \Omega_\mathcal{C}^1(\sum z_i), \CO_{\mathcal{C}}) \to \Hom\left(\Omega_{\mathcal{C}}^1(\sum z_i),\CO_{\mathcal{C}}\right)\\
\to H^0 (\mathcal{C}, f^*T\CX) &\to {\rm \mathbb{E}xt}^1(f^*\Omega_{\CX}^1 \to \Omega_\mathcal{C}^1(\sum z_i), \CO_{\mathcal{C}}) \to {\rm Ext}^1\left(\Omega_{\mathcal{C}}^1(\sum z_i),\CO_{\mathcal{C}}\right)\\
\to H^1 (\mathcal{C}, f^*T\CX) &\to {\rm \mathbb{E}xt}^2(f^*\Omega_{\CX}^1 \to \Omega_\mathcal{C}^1(\sum z_i), \CO_{\mathcal{C}}) \to 0.
\end{array}
\end{equation}
Here, the last part of the long exact sequence is zero because ${\rm Ext}^2\left(\Omega_{\mathcal{C}}^1(\sum z_i),\CO_{\mathcal{C}}\right) = 0$ for $\mathcal{C}$ a curve.
The first column of \eqref{eq:tan_obs} is the space of infinitesimal deformation and obstruction to deforming of $f$, and the third column is the space of infinitesimal automorphism and deformation of the domain $(\mathcal{C}, \vec{z})$.

In our case, $\CX$ is $\PB$ and $\mathcal{C}$ is $[ E_{\tau} / \Z_2 ]$ for some $\tau \in \mathbb{H}$.
Note that $f^* T\CX \cong T\mathcal{C}$ since $f : \mathcal{C} \to \CX$ is an orbifold covering map.
From the formula for the first Chern number of desingularized bundles \cite[Prop. 4.2.1]{CR1}, $c_1 (|f^* T\CX|) = c_1 (f^* T\CX) - 4 \cdot\frac{1}{2} = -2$. The desingularized bundle $|f^*T\CX|$ is a complex vector bundle over $|\PB| \cong \PP^1$, hence $|f^* T\CX| \cong \CO_{\PP^1}(-2)$.
Since $f^*T\CX$ and $|f^*T\CX|$ have the same local holomorphic sections, $H^0 (\mathcal{C}, f^*T\CX ) = 0$ and $H^1 (\mathcal{C}, f^*T\CX) \cong \C$.

Secondly, note that ${\rm \mathbb{H}om}(f^*\Omega_{\CX}^1 \to \Omega_\mathcal{C}^1(\sum z_i), \CO_{\mathcal{C}})$ is the infinitesimal automorphism group of the orbifold stable maps, which is zero from by the definition of stability.
Also, ${\rm \mathbb{E}xt}^1(f^*\Omega_{\CX}^1 \to \Omega_\mathcal{C}^1, \CO_{\mathcal{C}})$ is the Zariski tangent space of the moduli space $\OL\CM_{0,4,d}(\CX, \vec{\mathbf{x}})$.
Since this moduli space is zero dimensional for positive $d$ (by our direct counting in Section \ref{sec:mainsec_sphcount}), its tangent space ${\rm \mathbb{E}xt}^1(f^*\Omega_{\CX}^1 \to \Omega_\mathcal{C}^1(\sum z_i), \CO_{\mathcal{C}})$ is zero.

Above long exact sequence can be rewritten as
\begin{equation*}\label{eqn:lesreg}
\begin{array}{rcl}
0 \to & 0 &\to \Hom\left(\Omega_{\mathcal{C}}^1(\sum z_i),\CO_{\mathcal{C}}\right)\\
\to 0 \to & 0 &\to {\rm Ext}^1\left(\Omega_{\mathcal{C}}^1(\sum z_i),\CO_{\mathcal{C}}\right)\\
\stackrel{\delta}{\to} \C \to & {\rm \mathbb{E}xt}^2(f^*\Omega_{\CX}^1 \to \Omega_\mathcal{C}^1(\sum z_i), \CO_{\mathcal{C}}) &\to 0
\end{array}
\end{equation*}

To prove the regularity of our maps, we claim that ${\rm \mathbb{E}xt}^2(f^*\Omega_{\CX}^1 \to \Omega_\mathcal{C}^1(\sum z_i), \CO_{\mathcal{C}}) = 0$.
From Riemann-Roch formula,
$\dim{\rm Ext}^1\left(\Omega_{\mathcal{C}}^1(\sum z_i),\CO_{\mathcal{C}}\right) - \dim\Hom\left(\Omega_{\mathcal{C}}^1(\sum z_i),\CO_{\mathcal{C}}\right) = 3g-3+4 = 1$.
Hence, the injective map $\delta : {\rm Ext}^1\left(\Omega_{\mathcal{C}}^1(\sum z_i),\CO_{\mathcal{C}}\right) = \C \to \C$ should be surjective too, and this proves the claim.

\begin{remark}[Constant map]\label{rmk:appconst1}
The above argument can not apply to a constant map contributing to $\langle \Delta_i,\Delta_i,\Delta_i,\Delta_i\rangle_{0,4,d=0}$. The problem is that $\dim \OL\CM_{0,4,d=0}(\CX, \vec{\mathbf{x}})$ is no longer zero dimensional. (There is at least one dimensional parameter on this moduli from the domain complex structures.)
Hence, the moduli has a bigger dimension than the expected one, or equivalently, these constant maps have nontrivial obstructions, 
and this explains why they give rise to a negative rational number in the Gromov-Witten potential.
\end{remark}

\section{Orbifold covering theory}\label{app:orbcovth}
An orbifold $\CX$ is a Hausdorff space $|\CX|$ which locally looks as follows:
\begin{definition}
A local uniformizing chart $(\WT{V}, G, \phi)$ over an open set $V \subset |\CX|$ is a triple of following data :
\begin{enumerate}
	\item A finite group $G$ acts homeomorphic on an open set $\WT{V} \subset \R^n$.
	\item $\phi : \WT{V} \to V$ is a continuous map which factors through its orbit space $\WT{V}/G$ and induces a homeomorphism $\OL\phi : \WT{V}/G \to V$.
\end{enumerate}
\end{definition}

Similarly to the definition of manifold, these local charts are glued together in a compatible way to form an orbifold. A morphism between two orbifolds $\CX$ and $\CY$ is locally a $(G,G')$-equivariant map $(\WT{V}, G) \to (\WT{V}', G')$, where $(\WT{V}, G)$ and $(\WT{V}', G')$ are uniformizing charts for $\CX$ and $\CY$ respectively. (Here, a group homomorphism $G \to G'$ is contained in the data of the morphism.)  We refer readers to \cite{ALR} for more details on orbifolds.

We briefly recall the notion of orbifold covering maps following \cite{T}.
\begin{definition}
For two orbifolds $\CX$ and $\CY$, an orbifold morphism $p : \CX \to \CY$ is an orbifold covering map if it satisfies the following conditions:
\begin{enumerate}
	\item the continuous map $|p| : |\CX| \to |\CY|$ induced from $p$ is a surjective map;
	\item for each $y \in |\CY|$, there is a local uniformizing chart $(\WT{V}_y, G_y, \phi_y)$ of $y$ such that each point $x \in |p|^{-1}(y)$ has a local uniformizing chart $(\WT{V}_x, G_{y,x}, \psi_x)$ for some $G_{y,x} \leq G_y$ such that the following diagram commutes
	\begin{equation*}
	\xymatrix{
			{} 		& \WT{V}_y / G_{y,x} 	\ar[rr]^{\OL\psi_x : \cong}  \ar[d]^{q} 	&& U_{y,x} \ar[d]^{|p|}\\
		\WT{V}_y \ar[r] \ar[ru]	& \WT{V}_y / G_y	\ar[rr]^{\OL\phi_y : \cong}	&& V_y
	}
	\end{equation*}
where $q$ is a natural quotient map.
\end{enumerate}
\end{definition}

We have a lifting theorem for orbifold covering map which is similar to the ordinary one.
\begin{prop}\cite[Proposition 2.7]{T}
Let $f : ({\CX}, x) \to ({\CY}, y)$ be an orbi-map and $p : ({\CY}', y') \to ({\CY}, y)$ be an orbifold covering map. Then $f$ can be lifted to an orbi-map $\WT{f} : {\CX} \to {\CY}'$ if and only if $f_* \pi^{orb}_1 (\CX, x) \subset p_* \pi^{orb}_1 ({\CY}', y')$.
\end{prop}

There is one subtle point when applying this lifting theorem to holomorphic orbi-spheres in our example. While the map between orbifolds in the proposition are ``orbi-maps" \cite[Section 2]{T}, orbifold Gromov-Witten theory usually deals with ``good maps". However, for morphisms between 2-dimensional orbifolds, we have shown in \cite[Lemma 2.10]{HS}  that these two notions coincide.

\bigskip
Let $u : \mathcal{C} \to \CX$ be an orbifold stable map where $\mathcal{C}$ is an orbifold Riemann surface.
For each marked point $z_i$ in $\mathcal{C}$, choose a local uniformizing chart $(\WT{V}, G_{u(z_i)}, \phi)$ of $\CX$ near $u(z_i)$.
Then we take a local uniformizing chart $(\WT{U}_{z_i}, \Z_{m_i}, br_{m_i})$ near $z_i$ such that $u$ maps $U_{z_i}$ into $V$. Here, $\WT{U}_{z_i}$ is a holomorphic disc in a complex plane, $\Z_{m_i} \cong \left\langle \zeta_{m_i} := \exp\left(\frac{2\pi\sqrt{-1}}{m_i} \right) \right\rangle$ acts on $\WT{U}_{z_i}$ via left multiplication, and $br_{m_i} : z \mapsto z^{m_i}$ is a branched map of order ${m_i}$.
By the definition of orbifold morphism, there is a group homomorphism $u_\# : \Z_{m_i} \to G_{u(z_i)}$ and a $u_{\#}$-equivariant lifting $\WT{u}$ of $u$ near $z_i$ such that the following diagram commutes:
\begin{equation*}
\xymatrix{
\WT{U}_{z_i} \ar[r]^{\WT{u}} \ar[d]_{br_{m_i}} & \WT{V} \ar[d]^{\phi}	\\
U_{z_i} \ar[r]_u & V
}.
\end{equation*}
We denote $u_\# (\zeta_{m_i}) \in G_{u(z_i)}$ by $g_i$.
Note that the local isotropy group $G_{u(z_i)}$ is well-defined up to conjugacy.
For each conjugacy class $(g_i)$, we associate a ration number $\iota_{(g_i)}$ which is defined as follows.
Consider a linearized action of $g_i$ on the tangent space $T_{u(z_i)}\WT{V}$.
Choosing a metric which is invariant under the action of local isotropy group, the linearized action can be written as a diagonal matrix of form
\begin{equation*}
	{\rm diag} \left(\exp\left(\frac{2\pi\sqrt{-1}m_{i,1}}{m_i}\right), \cdots, \exp\left(\frac{2\pi\sqrt{-1}m_{i,n}}{m_i}\right)\right)
\end{equation*}
for some $m_{i,j} \in \{0,1,\cdots,m_i -1\}$.
We define the {\it degree shifting number} (or {\it age}) associated with the conjugacy class $(g_i)$ as
$$\iota_{(g_i)} := \sum_{j = 1}^n \frac{m_{i,j}}{m_i}.$$
Because the above diagonal matrix is invariant under conjugate action of local isotropy groups, the degree shifting number is well-defined.

Consider the moduli space $\OL\CM_{g,n,\beta}(\CX)$ of orbifold stable maps from genus-0 orbifold Riemann surface (with nodal singularities) with $n$-marked points to $\CX$ whose image represent a homology class $\beta \in H_2 (\CX, \Q)$.
It can be decomposed into several components $\OL\CM_{g,n,\beta}(\CX, \mathbf{g})$parametrized by $n$-tuple $\mathbf{g} = (g_1, \cdots g_n)$ of conjugacy classes $(g_i)$.
Using Kuranishi perturbation technique, Chen and Ruan \cite{CR2} showed that these moduli spaces admit virtual fundamental classes
$$[\OL\CM_{g,n,\beta}(\CX, \mathbf{g})]^{vir} \in H_*(\OL\CM_{g,n,\beta}(\CX, \mathbf{g}), \Q)$$
where $\ast=2c_1(T\CX)\cap \beta + 2(dim_{\C}\CX-3)(1-g)+2n-2\sum_{i=1}^n\iota_{(g_i)}$.

\bigskip
Let us consider the case of $\CX = \PB$.
We prove that any nonconstant holomorphic orbi-spheres contributing to the $4$-point Gromov-Witten invariants is an orbifold covering map.
Note that for any map $u$ contributing to $\WT{\mathcal{M}}_{0,4,d} (\D{1} , \D{j}, \D{k}, \D{l})$, the degree shifting number $\iota_{(g_i)}$ at each marked point is $\frac{1}{2}$.
Hence
\begin{equation*}\label{eq:vdimzero}
\sum_{i=1}^4 \iota_{(g_i)} = 2.
\end{equation*}

Now we give a proof of Lemma \ref{lem:RHformula}, which is a simple modification of the proof of Lemma 4.2 in \cite{HS}. For simplicity, we will use $\mathcal{C}$ and $\CX$ to denote the domain and the target space of the map $u : \PB(\Lambda) \to \PB(\Lambda_0)$ (so, $u: \mathcal{C} \to \CX$).
\begin{proof}[Proof of Lemma \ref{lem:RHformula}]
As before, let $\{ z_1, \cdots, z_4\}$ and $\{ w_1, \cdots, w_4 \}$ be the orbifold marked points in $\mathcal{C}$ and the orbifold singular points in $\CX$, respectively.
Since $u$ maps $z_i$ in the domain to an orbifold singular point in $\CX$, there is a function $I : \{1, \cdots, 4\} \to \{1, \cdots, 4\}$ such that $u(z_i) = w_{I(i)} \in {\CX}_{(g_i)}$.
We denote the number of points in $u^{-1}(w_i) - \{z_1, \cdots, z_k \}$ by  $m(w_i)$.

Let $U_i$ be an open neighborhood of $z_i$ in $\mathcal{C}$ with a uniformizing chart $(\WT{U}_i, \Z_2, br_i)$ where $br_i : z \mapsto z^{2}$, and let $V$ be an open neighborhood of $w_{I(i)}$ uniformized by $(\WT{V}, \Z_{2}, br_i)$ for $br_i : z \mapsto z^{2}$.
There is a local holomorphic lifting $\WT{u}$ of $u$ such that
\begin{equation*}
\xymatrix{
\WT{U}_i \ar[r]^{\WT{u}} \ar[d]_{br_i} & \WT{V} \ar[d]^{br}	\\
U_i \ar[r]_u & V
}
\end{equation*}
commutes.
Then, from the equivariance with respect to $u_{\#} : \Z_{2} \stackrel{id}{\to} \Z_{2}$, $\WT{u} (z) = z^{2 a_i + 1}$ for some $a_i \in \N_{\geq 0}$.

Since any orbifold Riemann surface is analytically isomorphic to the underlying Riemann surface, the orbifold map $u$ can be regarded as a branched covering map from $\mathbb{P}^1$ to itself. For example, one can use a coordinate $\underline{z} = z^2$ locally around a marked point $z_i$ (where $z$ is a coordinate on $\WT{U}_i$). Then we have $u(\underline{z}) = \underline{z}^{2 a_i +1}$ downstairs. We will apply Riemann-Hurwitz formula to $u$ viewed as a map between $\mathbb{P}^1$.

The ramification index at $z_i$ is $2 a_i + 1$ for $i = 1, \cdots, 4$.
Except these orbifold singular points, other points in the inverse image $u^{-1} (w_i)$ for $i = 1, \cdots, 4$ has ramification index which is an even positive integer $2 e^i_j$ for some $e^i_j \in \N$ ($j = 1, \cdots, m(w_4)$).
The Riemann-Hurwitz formula for $u$ tells us that
\begin{align}\label{eq:RHur}
2 \leq 2d - \left\{ \sum_{i=1}^4 2 a_i + \sum_{i=1}^{4} \sum_{j=1}^{m(w_i)}(2 e_j^i - 1)  \right\}
\end{align}
where $d$ is the degree of $u$. (Here, $2$ in the left hand side is the topological Euler characteristic of $\C\mathrm{P}^1$.) If $u$ does not have any branching outside $\cup_{i=1}^4 u^{-1}(w_i)$, then the equality holds in \eqref{eq:RHur}. Since $d$ is the weighted count of the  number of points in the fiber $u^{-1} (w_i)$ of $u$, we have
\begin{equation}\label{eq:ddeginv}
d = \sum_{j \in I^{-1}(i)}(2 a_j + 1) + 2 \sum_{j = 1}^{m(w_i)} e^i_j
\end{equation}
for each $i = 1, \cdots, 4$.
Plugging \eqref{eq:ddeginv} into \eqref{eq:RHur}, we have inequality
\begin{equation}\label{eq:ddeginv1}
2d \leq 2 + \sum_{i = 1}^{4} m(w_i).
\end{equation}

Equation \eqref{eq:ddeginv} together with $e^i_j \geq 1$ implies that
\begin{align}\label{eq:ddeginv2}
d \geq \sum_{j \in I^{-1}(i)}(2 a_j + 1) + 2 m(w_i),
\end{align}
for each $i = 1, \cdots, 4$.
Combining \eqref{eq:ddeginv1} and \eqref{eq:ddeginv2},
\begin{equation*}\label{eq:main}
\begin{array}{l}
\sum_{i=1}^4 a_i \leq 0.
\end{array}
\end{equation*}
Therefore all $a_i$'s are zero from the non-negativity of $a_i$'s.

If we do not use the inequality \eqref{eq:ddeginv2} and proceed, we have the following more precise estimate
\begin{align*}
\sum_{i=1}^4 \sum_{j=1}^{m(i)} e^i_j \leq \sum_{i=1}^4 m(i)
\end{align*}
which implies $e^i_j = 1$ for all $i,j$. We conclude that $u$ is an orbifold covering map.

\end{proof}


\bibliographystyle{amsalpha}

\end{document}